\documentclass[11pt]{amsart}

\usepackage{amssymb}

\usepackage{graphics}

\usepackage{latexsym}

\usepackage{amsmath}

\usepackage{amsthm}

\usepackage{amssymb,amsthm,amsfonts}

\usepackage{amscd}

\usepackage[arrow, matrix, curve]{xy}

\usepackage{syntonly}

\usepackage{mathrsfs}

\usepackage{amssymb}

\usepackage{amsxtra}

\usepackage{graphics}

\usepackage{braket}

\usepackage{mathtools}

\usepackage[normalem]{ulem}

\usepackage{color}

\usepackage[square,sort,comma,numbers]{natbib}

\usepackage{tikz-cd}

\newcommand{\raju}[1]{\color{magenta}\textbf{{ (Raju: #1) }}\color{black}}

\title[ ]{Periodic de Rham bundles over curves}

\author[Raju Krishnamoorthy]{Raju Krishnamoorthy}

\author[Mao Sheng]{Mao Sheng}

\email{krishnamoorthy@uni-wuppertal.de}

\address{Department of Mathematics, Bergische Universit\"at Wuppertal, Fakult\"at Mathematik und Naturwissenschaften, Gau{\ss}  stra{\ss}e 20}

\email{msheng@ustc.edu.cn}

\address{School of Mathematical Sciences, University of Science and Technology of China, Hefei, 230026, China}

\begin{document}

	
	\theoremstyle{plain}
	
	\newtheorem{thm}{Theorem}[section]
	
	\newtheorem{theorem}[thm]{Theorem}
	
	\newtheorem{lemma}[thm]{Lemma}
	
	\newtheorem{sublemma}[thm]{Sublemma}
	
	\newtheorem{semisimple lemma}[thm]{Semisimplicity lemma}
	
	\newtheorem{periodic lemma}[thm]{Periodicity lemma}
	
	\newtheorem{corollary}[thm]{Corollary}
	
	\newtheorem{proposition}[thm]{Proposition}
	
	\newtheorem{observation}[thm]{Observation}
	
	\newtheorem{addendum}[thm]{Addendum}
	
	\newtheorem{variant}[thm]{Variant}
	
	
	\theoremstyle{definition}
	\newtheorem{setup}[thm]{Setup}
	
	\newtheorem{construction}[thm]{Construction}
	
	\newtheorem{perspective}[thm]{Perspective}
	
	\newtheorem{notations}[thm]{Notation}
	
	\newtheorem{notation}[thm]{Notation}
	
	\newtheorem{question}[thm]{Question}
	
	\newtheorem{problem}[thm]{Problem}
	
	\newtheorem{remark}[thm]{Remark}
	
	\newtheorem{remarks}[thm]{Remarks}
	
	\newtheorem{definition}[thm]{Definition}

	\newtheorem{claim}[thm]{Claim}
	
	\newtheorem{assumption}[thm]{Assumption}
	
	\newtheorem{assumptions}[thm]{Assumptions}
	
	\newtheorem{properties}[thm]{Properties}
	
	\newtheorem{example}[thm]{Example}
	
	\newtheorem{conjecture}[thm]{Conjecture}
	
	\numberwithin{equation}{thm}

	
	\newcommand{\pP}{{\mathfrak p}}
	
	\newcommand{\pH}{{\mathfrak h}}
	
	\newcommand{\pX}{{\mathfrak X}}
	
	\newcommand{\pY}{{\mathfrak Y}}
	
	\newcommand{\sA}{{\mathcal A}}
	
	\newcommand{\sB}{{\mathcal B}}
	
	\newcommand{\sC}{{\mathcal C}}
	
	\newcommand{\sD}{{\mathcal D}}
	
	\newcommand{\sE}{{\mathcal E}}
	
	\newcommand{\sF}{{\mathcal F}}
	
	\newcommand{\sG}{{\mathcal G}}
	
	\newcommand{\sH}{{\mathcal H}}
	
	\newcommand{\sI}{{\mathcal I}}
	
	\newcommand{\sJ}{{\mathcal J}}
	
	\newcommand{\sK}{{\mathcal K}}
	
	\newcommand{\sL}{{\mathcal L}}
	
	\newcommand{\sM}{{\mathcal M}}
	
	\newcommand{\sN}{{\mathcal N}}
	
	\newcommand{\sO}{{\mathcal O}}
	
	\newcommand{\sP}{{\mathcal P}}
	
	\newcommand{\sQ}{{\mathcal Q}}
	
	\newcommand{\sR}{{\mathcal R}}
	
	\newcommand{\sS}{{\mathcal S}}
	
	\newcommand{\sT}{{\mathcal T}}
	
	\newcommand{\sU}{{\mathcal U}}
	
	\newcommand{\sV}{{\mathcal V}}
	
	\newcommand{\sW}{{\mathcal W}}
	
	\newcommand{\sX}{{\mathcal X}}
	
	\newcommand{\sY}{{\mathcal Y}}
	
	\newcommand{\sZ}{{\mathcal Z}}
	
	\newcommand{\sa}{{\bf a}}

	\newcommand{\scrA}{{\mathscr A}}
	
	\newcommand{\scrD}{{\mathscr D}}
	
	\newcommand{\scrS}{{\mathscr S}}
	
	\newcommand{\scrU}{{\mathscr U}}
	
	\newcommand{\scrL}{{\mathscr L}}
	
	\newcommand{\scrM}{{\mathscr M}}
	
	\newcommand{\scrX}{{\mathscr X}}
	
	\newcommand{\scrE}{{\mathscr E}}
	
	\newcommand{\scrY}{{\mathscr Y}}
	
	\newcommand{\scrV}{{\mathscr V}}
	
	\newcommand{\scrF}{{\mathscr F}}
	
	
	\newcommand{\A}{{\mathbb A}}
	
	\newcommand{\B}{{\mathbb B}}
	
	\newcommand{\C}{{\mathbb C}}
	
	\newcommand{\D}{{\mathbb D}}
	
	\newcommand{\E}{{\mathbb E}}
	
	\newcommand{\F}{{\mathbb F}}
	
	\newcommand{\G}{{\mathbb G}}
	
	\newcommand{\HH}{{\mathbb H}}
	
	\newcommand{\I}{{\mathbb I}}
	
	\newcommand{\J}{{\mathbb J}}
	
	\renewcommand{\L}{{\mathbb L}}
	
	\newcommand{\K}{{\mathbb K}}
	
	\newcommand{\M}{{\mathbb M}}
	
	\newcommand{\N}{{\mathbb N}}
	
	\renewcommand{\P}{{\mathbb P}}
	
	\newcommand{\Q}{{\mathbb Q}}
	
	\newcommand{\R}{{\mathbb R}}
	
	\newcommand{\SSS}{{\mathbb S}}
	
	\newcommand{\T}{{\mathbb T}}
	
	\newcommand{\U}{{\mathbb U}}
	
	\newcommand{\V}{{\mathbb V}}
	
	\newcommand{\W}{{\mathbb W}}
	
	\newcommand{\X}{{\mathbb X}}
	
	\newcommand{\Y}{{\mathbb Y}}
	
	\newcommand{\Z}{{\mathbb Z}}

	\newcommand{\id}{{\rm id}}
	
	\newcommand{\rank}{{\rm rank}}
	
	\newcommand{\pdeg}{{\rm pardeg}}
	
	\newcommand{\END}{{\mathbb E}{\rm nd}}
	
	\newcommand{\End}{{\rm End}}
	
	\newcommand{\Hom}{{\rm Hom}}
	
	\newcommand{\Hg}{{\rm Hg}}
	
	\newcommand{\tr}{{\rm tr}}
	
	\newcommand{\Sl}{{\rm Sl}}
	
	\newcommand{\Gl}{{\rm Gl}}
	
	\newcommand{\Gr}{{\rm Gr}}
	
	\newcommand{\Cor}{{\rm Cor}}
	
	\newcommand{\Res}{\mathrm{Res}}
	
	\newcommand{\HIG}{\mathrm{HIG}}
	
	\newcommand{\PHIG}{\mathrm{PHIG}}
	
	\newcommand{\MHG}{\mathrm{MHG}}
	
	\newcommand{\PHG}{\mathrm{PHG}}
	
	\newcommand{\ppar}{\mathrm{par}}
	
	\newcommand{\PaHG}{\mathrm{HG}^{\text{par}}}
	
	\newcommand{\PPHG}{\mathrm{PHG}^{\text{par}}}
	
	\newcommand{\LPHG}{\mathrm{PHG}^{\text{triv. par}}}
	
	\newcommand{\MIC}{\mathrm{MIC}}
	
	\newcommand{\DR}{\mathrm{DR}}
	
	\newcommand{\MDR}{\mathrm{MDR}}
	
	\newcommand{\PDR}{\mathrm{PDR}}
	
	\newcommand{\PaDR}{\mathrm{DR}^{\text{par}}}
	
	\newcommand{\PPDR}{\mathrm{PDR}^{\text{par}}}
	
	\newcommand{\LPDR}{\mathrm{PDR}^{\text{triv. par}}}
	
	\newcommand{\HDF}{\mathrm{HDF}}
	
	\newcommand{\DHF}{\mathrm{DHF}}
	
	\newcommand{\Spec}{\mathrm{Spec\ }}
	
	\newcommand{\Char}{{\rm char}}
	
	
	\thanks{This research work is supported by the National Key R and D Program of China 2020YFA0713100, CAS Project for Young Scientists in Basic Research Grant No. YSBR-032, National Natural Science Foundation of China (Grant No. 11721101), the Fundamental Research Funds for the Central Universities (No. WK3470000018) and Innovation Program for Quantum Science and Technology (2021ZD0302904).}
	
	\begin{abstract} In this article, we introduce the notion of periodic de Rham bundles over smooth complex curves. We prove that motivic de Rham bundles over smooth complex curves are periodic. We conjecture that irreducible periodic de Rham bundles over smooth complex curves are motivic. We show that the conjecture holds for rank one objects and rigid objects. 
	\end{abstract}
	\maketitle
	\tableofcontents
	\section{Introduction}
	The periods associated to a family of algebraic varieties satisfy a linear differential equation, a so-called \emph{Picard-Fuchs equation}. It is a fundamental problem in mathematics to characterize Picard-Fuchs equations among all linear differential equations. C. Simpson has proposed two conjectures, which aim  to provide sufficient conditions for a linear differential equation to be Picard-Fuchs: the \emph{standard conjecture} \cite[p. 372]{Sim90}, which is based on transcendence theory, and a \emph{motivicity conjecture} for rigid local systems based on his work in nonabelian Hodge theory \cite[p. 9]{S92}.
	Our investigations have been guided by these conjectures. In this article, we propose a new such characterization, which is measured by the dynamical properties of a twisted Frobenius action on a linear differential equation modulo almost all primes. The key new ingredient comes from the nonabelian Hodge theory in positive characteristic, due to Ogus-Vologodsky \cite{OV}.  We shall extend the theory to the parabolic setting. 
	
	Let us start with the basic setup. Let $C$ be a connected smooth projective curve over $\C$ and $U\subset C$ a nonempty Zariski open subset. In this article, a \emph{Gau{\ss}-Manin system} over $U$ means the following: Let $f\colon X\rightarrow U$ be a \emph{smooth projective} morphism of relative dimension $d\geq 0$. For each degree $0\leq i\leq 2d$, one may associate the following objects:
	\begin{itemize}
		\item [(i)] $V:=H^{i}_{dR}(X|U)$, which is the $i$-th hypercohomology of the relative de Rham complex of $f$;
		\item [(ii)] $\nabla:=\nabla^{GM}: V\to V\otimes\Omega_U$, the Gau{\ss}-Manin connection; and
		\item [(iii)] $Fil:=F_{hod}$, the Hodge filtration on $V$. 
	\end{itemize}
	A \emph{Gau{\ss}-Manin system }is the triple $(H^{i}_{dR}(X|U),\nabla^{GM},F_{hod})$ for some $f$ and some $i$. It is a standard fact in Hodge theory that (a) $V$ is a vector bundle over $U$, (b) $\nabla$ is an integrable connection (since $U$ is one-dimensional, this condition becomes vacuous here), (c) $Fil=Fil^{\bullet}$ is a finite decreasing effective filtration of subbundles and satisfies Griffiths transversality
	$$
	\nabla: Fil^l\to Fil^{l-1}\otimes \Omega_U.
	$$ 
	We call such a filtration a \emph{Hodge} filtration. These properties are purely algebraic. We make them into a terminology as follows.
	\begin{definition}
		A triple $(V,\nabla,Fil)$ over $U$ satisfying the above properties (a)-(c) is called a \emph{de Rham bundle}. A de Rham bundle is said to be \emph{motivic} if over the generic point $\xi$ of $U$, it is isomorphic to a subsystem of a Gau{\ss}-Manin system. 
	\end{definition}
	A further explanation is useful: a subsystem of a Gau{\ss}-Manin system over $\xi$ is nothing but a $k(C)$-subspace of $H^{i}_{dR}(X|U)\otimes \sO_{\xi}$ invariant under the $\nabla^{GM}\otimes \sO_{\xi}$ action, which is equipped with the induced filtration from $F_{hod}\otimes \sO_{\xi}$. 
	
	Our goal is to characterize motivic de Rham bundles among all de Rham bundles. It bears a natural relation with the original problem. A linear differential equation over $U$ is nothing but a holomorphic connection $\nabla$ on a holomorphic vector bundle $V$ over $U$. If it is Picard-Fuchs, then $\nabla$ has regular singularities at the infinity and the local monodromies are quasi-unipotent.  By Deligne's canonical extension \cite{Del70}, $(V,\nabla)$ is in fact algebraic, so one may in fact start with with an algebraic connection. Furthermore, we shall see that there is a natural Hodge filtration on $V$ which is often \emph{the} right one if $(V,\nabla)$ is Picard-Fuchs, a phenomenon which seems to have not been observed before (see Remark \ref{simpson filtration}). Now let us introduce the key notion of the article.  
	
	Let $\sU$ be a \emph{spreading-out} of $U$ over $\Z$, which means the following: $\sU$ is a smooth scheme over an affine scheme $S=\Spec(A)$ with $A\subset \C$ finitely generated $\Z$-subalgebra, together with an isomorphism $\alpha:\sU\times_A\C\cong U$. 
	Spreading out is a useful technique for algebro-geometric objects of finite presentation \cite{EGA IV}; and one of the main utilities is that it allows us to speak of reduction modulo $p$. Now let $(V,\nabla,Fil)$ be a de Rham bundle over $U$. A spreading-out of the quadruple $(U,V,\nabla,Fil)$ is a spreading-out $\sU$ of $U$ together with a triple $(\sV,\nabla,\mathcal Fil)$ over $\sU$, where $\sV$ is a vector bundle over $\sU$ (=locally free $\sO_{\sU}$-module of finite rank), $\nabla$ is an integrable $S$-relative connection 
	$$
	\nabla: \sV\to \sV\otimes \Omega_{\sU/S},
	$$ 
	and $\mathcal Fil$ is a (Griffiths transverse) Hodge filtration on $\sV$, together with an isomorphism, compatible with $\alpha$, 
	$$
	\beta: (\sV,\nabla,\mathcal Fil)\times_A\C\cong (V,\nabla,Fil).
	$$
	\begin{definition}\label{definition of motivic and periodic objects}
		Let $U$ be a complex smooth curve. A de Rham bundle $(V,\nabla,Fil)$ over $U$ said to be \emph{periodic} if there exists
		\begin{itemize}
			\item a positive integer $f$;
			\item a spreading-out  $(\sU,\sV,\nabla,\mathcal Fil)$ of $(U,V,\nabla,Fil)$ defined over $S=\Spec A$, where $A\subset \C$ is finitely generated over $\mathbb{Z}$; and
			\item a proper closed subscheme $Z\subset S$, 
		\end{itemize}
		such that the reduction $(\sV,\nabla,\mathcal Fil)\times_Ss$ for \emph{any} geometric point $s\in S-Z$ is periodic of period no greater than $f$ with respect to \emph{any} $W_2(k(s))$-lifting $\tilde s\subset S$.
	\end{definition}
	
	We emphasize that the above definition is independent of the choice of spreading-out. The periodicity of a de Rham bundle in positive characteristic refers to \cite[Definition 4.1]{SZ20}, whose Higgs analogue was introduced earlier in the work \cite{LSZ}. We shall give a further study of periodic de Rham bundles in positive characteristic in \S3. Our main result of the article is the following \emph{de Rham periodicity theorem}.
	\begin{theorem}\label{main result}
		Let $U$ be a smooth complex curve. Any motivic de Rham bundle over $U$ is periodic.
	\end{theorem}
  To appreciate the theorem, we point out that periodicity is a natural generalization of  ``being torsion'' in rank one case to an arbitrary rank. Theorem \ref{main result} has several immediate arithmetic consequences, which we choose to explore in a separate work. 
We conjecture that the converse statement of Theorem \ref{main result} also holds.
	\begin{conjecture}[Periodic de Rham conjecture]\label{motivic conjecture}
		Let $U$ be a smooth complex curve. Then any irreducible periodic de Rham bundle over $U$ is motivic.
	\end{conjecture}
	The irreducible condition in the conjecture cannot be dropped. One may show that there exist infinitely many nontrivial extensions of $(\sO_{\P^1},d)$ by $(\sO_{\P^1},d)$ over $\P^1$ with three punctures which are periodic. By the solution of Grothendieck-Katz $p$-curvature conjecture for rank one objects, due to Chudnovsky-Chudnovsky and Andr\'{e} (see e.g. \cite[\S2.4]{Bost}), we obtain a classification of rank one periodic objects.
\begin{proposition}[Proposition \ref{rank one case}]\label{rank one}
	A rank one de Rham bundle $(L,\nabla,Fil)$ is periodic if and only if there exists a positive integer $m$ such that $(L,\nabla)^{\otimes m}$ is trivial. 
\end{proposition}
It follows from Proposition \ref{rank one} that, the conjecture holds for rank one objects. We explain that the conjecture also holds true for \emph{rigid objects}.  This arises from the study of rigid local systems by N. Katz \cite{Kat96}. Set $D=C-U=\{P_1,\cdots,P_l\}$ to be the set of punctures. Let $(V,\nabla)$ be an  algebraic connection over $U$. Fix a set of Jordan normal forms $J:=\{J_1,\cdots,J_l\}$. If the set of residues 
	$$
	\Res(\nabla)=(\Res_{P_1}\nabla,\cdots,\Res_{P_l}\nabla)
	$$ 
	is \emph{component-wise} conjugate to $J$, then we write $\Res(\nabla)\cong J$. Notice that if the set of residues of $\nabla$ conjugates to $J$, so does  the set of residues of the tensor product connection $(V,\nabla)\otimes (L,\nabla)|_{U}$ for any rank one connection $(L,\nabla)$ over $C$. Because of this, an algebraic connection $(W,\nabla)$ is said to be \emph{twisted isomorphic} to $(V,\nabla)$ if there exists some $(L,\nabla)$ over $C$ such that $(W,\nabla)\cong (V,\nabla)\otimes (L,\nabla)|_U$. We denote this relation by 
	$$
	(W,\nabla)\thicksim (V,\nabla).
	$$
	Consider the following twisted de Rham moduli space
	$$
	M^{r,J}_{dR}(U)_{twisted}:=\{(V,\nabla)| \rank(V)=r,\ \Res(\nabla)\cong J\}/\thicksim.
	$$
	The next definition is a reformulation of Katz's notion of weakly physically semi-rigid local systems \cite[\S1.2]{Kat96} via Riemann-Hilbert correspondence. 
	\begin{definition}\label{WPSR connection}
		Let $(V,\nabla)$ be an algebraic connection over $U$. It is said to be \emph{weakly physically semi-rigid} of type $J$, if the corresponding twisted de Rham moduli space consists of finitely many points.   
	\end{definition}
	
	We have the following result. 
	\begin{proposition}[Propositions \ref{irreducible WPSR is periodic}-\ref{motivicity of rigids}]
		Let $\{J_i\}_{1\leq i\leq n}$ be a set of Jordan normal forms whose eigenvalues are roots of unity. Let $(V,\nabla)$ be an irreducible weakly physically semi-rigid connection of type $J$ and \emph{torsion} determinant over $U$. The following statements hold:
		\begin{itemize}
			\item [(i)] The de Rham bundle $(V,\nabla,F_S)$, where $F_S$ is the restriction of the \emph{Simpson} filtration on the canonical parabolic extension of $(V,\nabla)$ over $C$, is periodic;
			\item [(ii)](Katz) $(V,\nabla,F_S)$ is motivic. 
		\end{itemize}
		Consequently, Conjecture \ref{motivic conjecture} holds for $(V,\nabla,F_S)$. 
	\end{proposition}
	By Theorem \ref{main result}, (ii) implies (i), and (ii) follows from a classification result of Katz \cite[Ch. 1, 8]{Kat96}. Inspired by the recent remarkable work of Esnault-Gr\"ochenig \cite{EG18} on the motivicity conjecture of Simpson, we shall provide a proof of (i) that is independent from the classification result. 

In Conjecture \ref{motivic conjecture}, when $U$ is defined over $\bar \Q$, one may view it as a de Rham to Higgs version of the standard conjecture \footnote{On the other hand, one may try to extend the standard conjecture to an arbitrary $U$.}. Moreover,  in the case of curve, one often has non-rigid irreducible connections. Hence Conjecture \ref{motivic conjecture} can be also viewed as a natural generalization of the motivicity conjecture for rigid connections. After completing the first version of this article, H. Esnault informed us the following \emph{nilpotency conjecture}:
\begin{conjecture}(\cite[Ch.V, Appendix]{Andre})\footnote{In loc. cit., Andr\'e considers only when $U$ is defined over $\bar \Q$, and assumes the nilpotency of $p$-curvature only for \emph{density one} primes (so it is even stronger than the global nilpotency condition).}
Let $(V,\nabla)$ be an irreducible algebraic connection over $U$. Suppose there is a spreading-out $(\sU,\sV, \nabla)$ of $(U,V,\nabla)$ over $S$ such that $(\sV,\nabla)$ is globally nilpotent \cite{Kat70}. Then $(V,\nabla)$ is motivic.  
\end{conjecture}
Since periodicity implies global nilpotency, the previous conjecture implies ours. Analogous to the fact that algebraic connections over $U$ whose spreading-outs are globally nilpotent form a neutral Tannakian category over $\C$, we prove in Proposition \ref{Tannakian cat} that the category of periodic de Rham bundles over $U$ also forms a neutral Tannakian category. Matching these two associated Tannaka groups is an interesting question. To each irreducible periodic de Rham bundle over $U$ (which is assumed to be defined over $\bar \Q$), we associate an essentially unique family of representations of the topological fundamental groups of $U'$, a finite \'etale cover of some nonempty open subset of $U$ (see the discussion after Conjecture \ref{extension conjecture}). To understand better this family of representations should be an important question. On the other hand, we also propose the periodic Higgs conjecture in \cite{KS}, the Higgs companion of Conjecture \ref{motivic conjecture}. The n\"{a}ive Higgs analogue of the nilpotency conjecture is however, simply false.   	

The notion of a periodic parabolic Higgs-de Rham flow (Definition 4.1) generalizes the notion of a periodic Higgs-de Rham flow of Lan-Sheng-Zuo \cite{LSZ} in a natural way. Some new phenomenon arises in the parabolic setting. However, its appearance in the current work is actually a \emph{necessity}. Indeed, in order to get the periodicity for a motivic de Rham bundle over $U$, we have to establish first the periodicity for its \emph{canonical parabolic extension} over $C$, the compactification of $U$, and then the periodicity carries over to $U$. The paper is structured as follows.
	\begin{itemize}
		\item In Section \ref{section:parabolic}, we generalize the Cartier/inverse Cartier transform of Ogus-Vologodsky \cite{OV} to the parabolic setting. This is achieved via a Biswas-Iyer-Simpson correspondence for $\lambda$-connections over an algebraically closed field, that is a generalization of the work \cite{Biswas} of Biswas and the work \cite{IS} of Iyer-Simpson.
		\item In Section 3, we introduce the notion of parabolic de Rham bundles.
		\item In Section \ref{section:periodic}, we introduce the notion of periodic (parabolic) de Rham bundles.
		\item In Section \ref{section:proof}, we prove our main result, Theorem \ref{main result}.
		\item In Section \ref{section:conjecture}, we discuss the periodic de Rham conjecture. 
	\end{itemize}    
	
We list some notations which are used throughout the article.

{\bf Notations.}
\begin{itemize}
	\item $k$ is an algebraically closed field. 
	\item $N$ is a positive integer, which is coprime to $p$ when $\Char(k)=p>0$.
	\item $G$ is the constant group scheme $\Z/N$ over $k$. 
	\item $\sigma$ is a generator of $G(k)$.
	\item $\zeta$ is a primitive $N$-th root of unity.
	\item $C$ is a connected smooth curve over $k$.
	\item $D=\sum D_i\subset C$ a reduced effective divisor with irreducible components $D_i$s. 
	\item $U=C-D$, the complement of $D$.
	\item $C_{\log}=(C,D)$, the logarithmic curve with the logarithmic structure determined by $D$.
	\item $\tilde C_{\log}=(\tilde C,\tilde D)$ in the case $\Char(k)=p$, a $W_2(k)$-lifting of $C_{\log}$.  
	\item $S$, a locally noetherian scheme. It may be regarded as a log scheme with the trivial logarithmic structure. 
	\item $Z$, a proper closed subscheme of $S$.
	\item $\sC$, a smooth curve over $S$, i.e. $\sC\to S$ a smooth morphism of relative dimension one.
	\item $\sD=\sum_i \sD_i$, an $S$-relative simple normal crossing divisor in $\sC$, with irreducible components $\sD_i$s.
	\item $\sC_{\log}=(\sC,\sD)$ the log scheme whose logarithmic structure is determined by $\sD$.  
\end{itemize}

	{\bf Acknowledgement.} The idea of considering Higgs bundles over $\C$ whose mod $p$ reductions are periodic with bounded period for almost all $p$s comes from a conversation of the second named author with Kang Zuo during the Pisa conference 'Fundamental Group in Arithmetic and Algebraic Geometry' in December 2013. The current work and \cite{KS} began during the academic visit of the first named author to School of Mathematical Sciences, USTC in March, 2018 and continued in March, 2019. He would like to thank the hospitality of the local institute. The second named author benefited from various discussions with Jianping Wang. In particular, he pointed out to us that the condition $N\leq p-1$ in the first version of Proposition \ref{Biswas correspondence for Higgs/flat bundles in positive char} can be relaxed to $(N,p)=1$. Also, Proposition \ref{change of par weights under inverse Cartier}, which was absent in the first version, is due to his observation. During the writing of the manuscript, the second named author has received warm encouragements from Luc Illusie, Carlos Simpson, Arthur Ogus, Shing-Tung Yau and Kang Zuo. We thank them heartily.

	\section{Nonabelian Hodge theorem in positive characteristic: the parabolic setting}\label{section:parabolic}
	Let us fix some notations for this section.
	\begin{notation}\label{basic notation} In this section, $\Char(k)=p$.   
		\begin{itemize}
			\item$\HIG^{lf}_{p-1}(C_{\log}/k)$ is the category of Higgs bundles over $C_{\log}$ which are nilpotent of exponent $\leq p-1$.  
			\item $\MIC^{lf}_{p-1}(C_{\log}/k)$ is the category of flat bundles over $C_{\log}$ whose $p$-curvatures and residues along all components of $D$ are nilpotent of exponent $\leq p-1$.
		\end{itemize}
	\end{notation}
	
	Here, the superscript $lf$ refers to the fact that we only consider vector bundles and not more general coherent sheaves. Recall the nonabelian Hodge theorem in characteristic $p$:
	
	\begin{theorem}[Ogus-Vologodsky \cite{OV}, Schepler \cite{S}]\label{OV correspondence}
		
		Notation as in \ref{basic notation}. Then there is an equivalence of categories
		
		$$
		\HIG^{lf}_{p-1}(C_{\log}/k)\xrightleftharpoons[\ C\ ]{\ C^{-1}\ }\MIC^{lf}_{p-1}(C_{\log}/k),
		$$
		
	\end{theorem}
	
	The functors $C$ and $C^{-1}$ depend on the choice of $\tilde C_{\log}$, and $C^{-1}$ (resp. $C$) is called the ``inverse Cartier transform" (resp. ``Cartier transform"). There is an elementary reconstruction of the inverse Cartier transform/Cartier transform via ``exponential twisting", see \cite{LSZ0}, and \cite[\S 6]{LSYZ} for the logarithmic case. The aim of this section is to generalize the construction of Cartier/inverse Cartier transform for certain parabolic flat/Higgs bundles in positive characteristic. 
	
	We start with the introduction of parabolic $\lambda$-connections over a more general base scheme.  Let $\sC/S$ be a smooth curve, and $\sD=\sum_i\sD_i\subset \sC$ an $S$-relative simple normal crossings divisor.  Let $j:\sU=\sC-\sD\hookrightarrow \sC$ be the natural inclusion.
	\begin{definition}\label{parabolic sheaf}
		A parabolic sheaf on $\sC_{\log}/S$ is a collection of relatively torsion free coherent sheaves $V_{\alpha}$ indexed by multi-indices $\alpha=(\alpha_1,\cdots,\alpha_l)$ with $\alpha_i\in \R$, together with inclusions of sheaves of $\sO_{\sC}$-modules $V_{\alpha}\hookrightarrow V_{\beta}$ for $\alpha\geq \beta$ (we say $\alpha\geq \beta$ if $\alpha_i\geq \beta_i$ for all $i$), subject to the following conditions: 
		\begin{itemize}
			\item [(i)] (flatness) Each $V_{\alpha}$ is flat over $S$;
			\item [(ii)] (support and normalization) $V_{\alpha+\delta^i}=V_{\alpha}(-\sD_i)$ and $V_{\alpha+\delta^i}\hookrightarrow V_{\alpha}$ is the natural inclusion. Here $\delta^i=(0,\cdots,1,\cdots,0)$ is the multi-index having one at the $i$-th component and zero at the other components;
			\item [(iii)] (left continuous) for any $\alpha$, there exists some $c>0$ such that for any multi-index $\epsilon$ with $0\leq \epsilon_i<c$ we have $V_{\alpha}=V_{\alpha-\epsilon}$;
			\item [(iv)] (discreteness of weights) The subset of $t\in [0,1)$ such that $Gr_{\sD_i,t}V\neq 0$ (defined in the next sentence) is discrete in $\Q$ for each $i$.  For $0\le t<1$, set $V_{\sD_i,t}$ to be image subsheaf of $V_{t\delta^i}$ in $V_{0}/V_{\delta^i}$. Then set $Gr_{\sD_i,t}V=V_{\sD_i,t}/V_{\sD_i,t+\epsilon}$ with $\epsilon>0$ small.
		\end{itemize}
	\end{definition}
	By (iv), there are finitely many values $t\in [0,1)\cap \Q$ such that $Gr_{\sD_i,t}$ is nonzero. We call any $\Z$-translate of one of these values a \emph{weight} of $V$ along the component $\sD_i$. In this article, we consider only parabolic structures in \emph{rational} weights.  
	\begin{example}\label{trivial para}
		Let $V$ be an $S$-flat relatively torsion free sheaf on $\sC$. We define the \emph{trivial parabolic structure} as follows: set 
		$$ 
		V_{\alpha}=V(-\sum_{i=1}^l\lceil\alpha_i\rceil \sD_i)
		$$
		for $\lceil \ \rceil$ the ceiling function. 
	\end{example}
	The next is an important variant of the previous example. 
	\begin{example}\label{rational divisor}
		Fix a multi-index $\alpha\in \Q^l$. Set 
		$$
		V(-\sum_{i=1}^l\alpha_i\sD_i)_{\beta}=V(\sum_{i=1}^{l}-a_i\sD_i),
		$$
		where each $a_i$ is the least integer such that $a_i\geq \alpha_i+\beta_i$. We denote the parabolic sheaf by $V(-\sum_{i=1}^l\alpha_i\sD_i)$. We call a parabolic sheaf a \emph{parabolic line bundle} when it is of form $L(-\sum_{i=1}^l\alpha_i\sD_i)$ for some invertible sheaf $L$ over $\sC$.
	\end{example}
	
	\begin{definition}
		A parabolic sheaf $V$ is a \emph{parabolic vector bundle} if, in a Zariski neighborhood of any point $x\in \sC$ there is an isomorphism between $V$ and a direct sum of parabolic line bundles.\footnote{In \cite{IS}, Iyer-Simpson call this a \emph{locally abelian} parabolic bundle.} A morphism of parabolic sheaves $\phi: V\to W$ is a collection of morphisms $\phi_{\alpha}: V_{\alpha}\to W_{\alpha}$ satisfying the obvious commutativity for $\phi_{\alpha}$ and $\phi_{\beta}$ once $\alpha\geq \beta$.  
	\end{definition} 
	\begin{definition}
		For a $\lambda\in\Gamma(S,\sO_S)$, a \emph{parabolic $\lambda$-connection} over $\sC_{\log}/S$ consists of a parabolic vector bundle $V$ over $\sC$ and a $\lambda$-connection
		$$
		\nabla: V\to V\otimes \omega_{\sC_{\log}/S},
		$$
		where  $\omega_{\sC_{\log}/S}$ is the sheaf of relative logarithmic K\"ahler differentials. More precisely, this means that for each $\alpha$, there is a $\lambda$-connection $\nabla_{\alpha}: V_{\alpha}\to V_{\alpha}\otimes \omega_{\sC_{\log}/S}$ and there is a commutative diagram of $\nabla_{\alpha}$ and $\nabla_{\beta}$ for $\alpha\geq \beta$. A parabolic $\lambda$-connection is said to be flat if $\nabla_{\alpha}\wedge \nabla_{\alpha}=0$ for each $\alpha$. 
	\end{definition}
	Regarding $\sO_{\sC}(\sD_i)$ as a subsheaf of $j_*\sO_{\sU}$, the exterior differential 
	$$
	d: j_*\sO_{\sU}\to j_*\sO_{\sU}\otimes j_*\Omega_{\sU}
	$$
	induces a $\lambda$-connection $\lambda d: \sO_{\sC}(\sD_i)\to \sO_{\sC}(\sD_i)\otimes \omega_{\sC_{\log}/S}$. Let $(V,\nabla)$ be a $\lambda$-connection. For $V_{\alpha-\delta_i}=V_{\alpha}\otimes \sO_{\sC}(\sD_i)$, we claim that 
	\begin{claim}\label{tensor product}
		\begin{equation*}\label{connection on parabolic bundle}
			\nabla_{\alpha-\delta_i}=\nabla_{\alpha}\otimes id+id\otimes\lambda d.
		\end{equation*}
	\end{claim}
	\begin{proof}
		To see this, we may argue locally and assume that the local defining equation of $\sD_i$ is given by $f=0$. Let $v$ be a local section of $V_\alpha$. Then
		$$
		\nabla_{\alpha}(v)=\nabla_{\alpha-\delta_i}(fv\otimes\frac{1}{f})=\lambda v\otimes \frac{df}{f}+f\nabla_{\alpha-\delta_i}(v\otimes \frac{1}{f}). 
		$$
		As $V_{\alpha-\delta_i}$ is locally free, it follows that 
		$$
		\nabla_{\alpha-\delta_i}(v\otimes \frac{1}{f})=\nabla_{\alpha}(v)\otimes\frac{1}{f}+v\otimes \lambda d(\frac{1}{f}).
		$$
		The claim is proved. 
	\end{proof}
	Therefore, a $\lambda$-connection on a parabolic vector bundle $V$ is \emph{fully determined} by $\nabla_0$ on $V_0$.  Also, as a matter of convention, a parabolic flat $\lambda$-connection is said to be a \emph{parabolic flat bundle} for $\lambda=1$, while a \emph{parabolic Higgs bundle} for $\lambda=0$.

	Let $(V,\nabla)$ be a parabolic $\lambda$-connection. Note that the composite morphism
	$$
	V_{\alpha}\stackrel{\nabla_{\alpha}}{\longrightarrow}V_{\alpha}\otimes \omega_{\sC_{\log}/S}\stackrel{id\otimes res_{\sD_i}}{\longrightarrow} V_{\alpha}\otimes \sO_{\sD_i}
	$$
	is $\sO_{\sD_i}$-linear and factors through the restriction $V_{\alpha}\to V_{\alpha}\otimes \sO_{\sD_i}$. The resulting endomorphism of $V_{\alpha}\otimes \sO_{\sD_i}$ is said to be \emph{residue} of $\nabla_{\alpha}$ along $\sD_i$ and is denoted by $res_{\sD_i}\nabla_{\alpha}$. Note that for a parabolic $\lambda$-connection $(V,\nabla)$ over $\sC_{\log}/S$, there exists a positive integer $N$ such that its weights belong to $\frac{1}{N}\Z$. The following definition is a variant of Definition 3.2 \cite{IS} \footnote{We remind the reader that for $S=\Spec\C$, our definition differs from \cite{IS} by a minus sign. This is basically because a parabolic structure in our definition is decreasing, instead of increasing as adopted in \cite{IS}.}.
	\begin{definition}\label{adjusted condition}
		Let $(V,\nabla)$ be a parabolic $\lambda$-connection over $\sC_{\log}/S$, it is said to be \emph{adjusted} if the weights belong to $\frac{1}{N}\Z$ for some $N$ which is invertible in $S$, and moreover $res_{\sD_i}\nabla_{\alpha}$ acts on $Gr_{\sD_i,\alpha_i}V\neq 0$ as an operator with the eigenvalue $\lambda\alpha_i$.   
	\end{definition} 
	
	Now we resume the setup in \ref{basic notation}. Let $\mathrm{HIG}^{\ppar}_{p-1,N}(C_{\log}/k)$ be the category of parabolic Higgs bundles over $C_{\log}/k$ which are nilpotent of exponent $\leq p-1$ and whose parabolic structures are supported in $D$ and whose weights are contained in $\frac{1}{N}\Z$. Let  $\mathrm{MIC}^{\ppar}_{p-1,N}(C_{\log}/k)$ be the category of \emph{adjusted} parabolic flat bundles over $C_{\log}/k$ whose $p$-curvatures as well as the nilpotent part of residues are nilpotent of exponent $\leq p-1$ and whose parabolic structures are supported in $D$ and whose weights are contained in $\frac{1}{N}\Z$. Using Example \ref{trivial para}, $\HIG_{p-1}^{lf}(C_{\log/k})$ (resp. $\MIC_{p-1}^{lf}(C_{\log}/k)$) is naturally a full subcategory of $\mathrm{HIG}^{\ppar}_{p-1,N}(C_{\log}/k)$ (resp. $\mathrm{MIC}^{\ppar}_{p-1,N}(C_{\log}/k)$). We shall extend the Cartier/inverse Cartier transform to the larger categories.  
	\begin{theorem}\label{ov correspondence for parabolic objects}
		Let $C_{\log}$ and $\tilde C_{\log}$ be as \ref{basic notation}. Then there is an equivalence of categories 
		$$
		\mathrm{HIG}^{\ppar}_{p-1,N}(C_{\log}/k)\xrightleftharpoons[\ C_{\ppar}\ ]{\ C_{\ppar}^{-1}\ }\mathrm{MIC}^{\ppar}_{p-1,N}(C_{\log}/k),
		$$
		which restricts to the Cartier/inverse Cartier transform in Theorem \ref{OV correspondence} over the objects with trivial parabolic structure. The functors $C_{\ppar}$ and $C_{\ppar}^{-1}$ depend only on the choice of $W_2$-lifting of the pair $(C,D)$. 
	\end{theorem}

	The strategy to prove the last theorem is to make the reduction to the trivial parabolic structure case by base change along a suitable cyclic cover.  Our first step is going to establish a Biswas-Iyer-Simpson correspondence for parabolic $\lambda$-connections in arbitrary characteristic, which is of independent interest. To this purpose, we make a bit more general setup as follows:
	
	\begin{notation}\label{parabolic_notation} Let $k$ be an algebraically closed field.  Let $\pi: C'\to C$ be a cyclic cover of order $N$ with branch divisor $D$.  Let $D'=\sum_{i=1}^lD_i'$ be the reduced divisor $(\pi^*D)_{red}$. Hence $\pi^*D=ND'$. Set $C'_{\log}$ to be the log pair $(C',D')$.
	\end{notation}
	
	A $G$-equivariant vector bundle over $C'$ is a vector bundle $E$ over $C'$ together with a family of isomorphisms
	$$
	\phi_g: g^*E\cong E,
	$$
	indexed by $g\in G$ which satisfies the cocycle condition $\phi_g\circ g^*\phi_h=\phi_{hg}$ for any $g,h\in G$. 
	The construction of parabolic pushforward in the next lemma is taken from Biswas \cite{Biswas} (in characteristic zero). 
	\begin{lemma}\label{parabolic pushforward}
		Notation as in \ref{parabolic_notation}. Let $E$ be a $G$-equivariant vector bundle over $C'$. For each $\alpha\in \Q^l$, set 
	$$
	V_{\alpha}:=(\pi_*(E\otimes \sO_Y(\sum_i[-\alpha_i N]D_i')))^{G}.
	$$
		Then the collection of $V_{\alpha}$s form a parabolic vector bundle $V$ over $C$. The parabolic structure of $V$ is supported on $D$ and its weights are contained in $\frac{1}{N}\Z$. The construction from $E$ to $V$ forms a functor $\pi_{\ppar*}$ from the category of $G$-equivariant vector bundles over $C'$ to the category of parabolic vector bundles over $C$, which is called \emph{parabolic pushforward}.
	\end{lemma}
	\begin{proof}
		Set $U'=\pi^{-1}U=C'-D'$. Then the restriction $\pi: U'\to U$ is finite \'etale and Galois. For any $\alpha$, $V_{\alpha}|_{U}$ is isomorphic to $(\pi_*E|_{U'})^G$ which is the Galois descent of $E_{U'}$, and the natural inclusion $V_{\alpha}\hookrightarrow V_{\beta}$ for $\alpha\geq \beta$ is the identity. So the support of the parabolic structure is contained in $D$. It is also clear that the weights of $V$ are elements of $\frac{1}{N}\Z$. As $V_{\alpha}$ is a subsheaf of $\pi_*(E\otimes \sO_Y(\sum_i[-\alpha_i N]D_i')$, it is torsion free. Since $C$ is a smooth curve, it is locally free (which can be seen more clearly by examining $V_{\alpha}$ locally as we shall do below). So it remains to show $V_{\alpha}$ is locally isomorphic to direct sum of parabolic line bundles.
		
		Let $x$ be a local coordinate at some open neighborhood $U_i$ of a component $D_i$, and $U'_i=\pi^{-1}U_i$ is defined by $y^N=x$. Shrinking $U_i$ if necessary, we may assume that $E$ is free over $U'_i$. We will argue that, after possibly shrinking $U_i$ further, $E$ admits a basis of eigenvectors for $G$. As $D'_i=\{y=0\}$ is the fixed point under the $G$-action, $G$ acts on the fiber $E_{D'_i}$. Since $(N,p)=1$ when $\Char(k)=p>0$, the polynomial $t^N-1$ in any case is separable in $k[t]$ and hence splits into product of linear factors of multiplicity one. It follows that there is eigen decomposition 
		$$
		E_{D'_i}=\bigoplus_jk(\xi_j)
		$$
		where $\xi_j$s are characters of $G$. Define the $G$-equivariant direct sum of line bundles
		$$
		F:=\bigoplus_j \sO_{U'_i}\otimes _kk(\xi_j).
		$$
		Let $u_0\colon E_{U'_i}\to F$ be any morphism which restricts over $D'_i$ to an isomorphism given by the previous eigen decomposition. Set $u=\sum_{g\in G}g^{-1}u_0\circ g$, which is $G$-equivariant. Now let $\tilde U'_i\subset U'_i$ be the open subset over which $u$ is an isomorphism. It is non-empty, because $u_0$ and $u$ coincide at $y$. It must be $G$-invariant, which implies that it is the inverse image of an open subset $\tilde U_i\subset U_i$. Thus, we may assume at the beginning that over $U'_i$, $E$ admits an $\sO_{U'_i}$-basis $\{e_j\}$ which are eigenvectors under the $G$-action.  Write $R=\Gamma(U_i,\sO_{U_i})$ and $R'=\Gamma(U'_i,\sO_{U'_i})$. Then $R'$ is a free $R$-module with basis $\{1,y,\cdots,y^{N-1}\}$. Then $\sigma(y)=\zeta y$, where $\zeta$ be a primitive 
		$N$-th root of unity, and $\sigma(e_j)=\zeta^{n_j}e_j$ for some $0\leq n_j\leq N-1$. Assume $\alpha_i=\frac{m_i}{N}$ for some $0\leq m_i\leq N-1$. Then it follows that $V_{\alpha}$ over $U_i$ is isomorphic to 
		$$
		\bigoplus_j\sO_{U_i}(\frac{m_i-n_j}{N}D_i)
		$$
		as parabolic bundles. 
		
		Finally, let $h: E_1\to E_2$ be a morphism of $G$-equivariant vector bundles over $C'$. To see that $h$ induces a natural morphism of parabolic vector bundles $\pi_{\ppar*}E_1\to \pi_{\ppar*}E_2$, it suffices to check that 
		$$
		\pi_*(h)^G: \pi_*(E_1)^G\to \pi_*(E_2)^G
		$$
		preserves the parabolic structures. The question is local around each $D_i$ and one may assume that $E_i, i=1,2$ is a direct sum of $G$-equivariant line bundles. But for line bundles, the assertion is clear. Thus the parabolic pushforward is a functor. The lemma is proved.
	\end{proof}
	Given a parabolic vector bundle $V$ on $C$ whose support is contained in $D$ and weights are contained in $\frac{1}{N}\Z$, for each $1\leq i\leq l$ we construct a filtration of coherent sheaves on $C'$ as follows: set $E^0_i=\pi^*(V_0\otimes \sO_C(D))$. Let $\{\alpha_i^1,\cdots,\alpha_i^{w_i}\}\subset (0,1)$ be the nonzero weights supported at $D_i$. Write $\alpha_i^j=\frac{m_i^j}{N}$ for $1\leq m_i^j\leq N-1$. We assume that $\alpha_i^1<\cdots<\alpha_i^{w_i}$. We inductively define for $0\leq j\leq w_i$ by elementary transformation
	$$
	E^{j+1}_{i}:=\ker(E^{j}_i\to Gr_{D_i,\alpha_i^j}V\otimes_{k(D_i)}\sO_{C'}(ND'_i)/m^{N-m_i^j+1}_{D'_i}),
	$$
	where $k(D_i)$ is the residue field at the closed point $D_i$, $m_{D'_i}$ is the maximal ideal of the closed point $D'_i$ and the morphism $k(D_i)\to \sO_{C'}(ND'_i)/m^{N-m_i^j+1}_{D'_i}$ is the composite of natural morphisms
	$$
	k(D_i)\stackrel{\pi^*}{\to}\sO_{C'}(ND'_i)/m^{N}_{D'_i}\to \sO_{C'}(ND'_i)/m^{N-m_i^j+1}_{D'_i}.
	$$ 
	By construction, we have the following sequence of coherent sheaves:
	$$
	E^{0}_i\supset E^1_i\supset \cdots\supset E^{w_i}_i\supset E^{w_i+1}_i.
	$$
	As $E^0_i=\pi^*V_0\otimes \sO_{C'}(ND')$ for any $i$, $\{E^{w_i+1}_i\}_{1\leq i\leq l}$ are all subsheaves of the same sheaf. Define $E=\cap_{i=1}^{l}E^{w_i+1}_i$. Away from $D'$, $E$ is equal to $\pi^*V_0$; around each $D'_i$, $E$ is equal to $E^{w_i+1}_i$. As $E$ is torsion free, it is a vector bundle (which can be seen more clearly by local computation).

	The following lemma is a slight generalization of Lemma 2.3 \cite{IS} in the curve case. 
	\begin{lemma}\label{Biswas correspondence in positive char}
		The vector bundle $E$ over $C'$ constructed above from a parabolic vector bundle is $G$-equivariant. Furthermore, it forms a functor $\pi^*_{\ppar}$ from the category of parabolic vector bundles over $C'$, whose supports are contained in $D$ and weights are contained in $\frac{1}{N}\Z$, to the category of $G$-equivariant vector bundles over $C'$, which is called \emph{parabolic pullback}. The parabolic pushforward is an equivalence of categories and the parabolic pullback is a quasi-inverse to the parabolic pushforward.  
	\end{lemma}  
	\begin{proof}
		Since each $D_i$ is fixed under the $G$-action and $E_i^0$ is $G$-equivariant, each $E_i^j$ for $0\leq j\leq w_i$ is $G$-equivariant by construction. So is the intersection $E$. We are going to show that the parabolic pushforward is fully faithful and essentially surjective. Let $E_i, i=1,2$ be two $G$-equivariant vector bundles over $C'$. To verify that the natural map
		$$
		\Hom_{G}(E_1,E_2)\to \Hom(\pi_{\ppar*}E_1,\pi_{\ppar*}E_2), \quad h\mapsto \pi_{\ppar*}(h)
		$$
		is bijective, one may reduce it to the case of direct sums of $G$-equivariant line bundles, in which the verification becomes obvious. Indeed, we embed $\Hom_{G}(E_1,E_2)$ to $\Hom_{G}((E_1)|_{U'},(E_2)|_{U'})$. By Galois descent, the natural map is bijective:
		$$
		\pi_*^G: \Hom_{G}((E_1)|_{U'},(E_2)|_{U'})\to  \Hom((\pi_{\ppar*}E_1)|_{U},(\pi_{\ppar*}E_2)|_{U}).
		$$
		So the problem is local around each $D_i$, and as argued in Lemma \ref{parabolic pushforward}, one may fix an $G$-equivariant isomorphism of $E_i, i=1,2$ to a direct sum of $G$-equivariant line bundles. To show the essentially surjectivity of the functor $\pi_{\ppar*}$, we are going to verify that there is a natural isomorphism of parabolic vector bundles:
		$$
		V\cong \pi_{\ppar*}\pi_{\ppar}^*V,
		$$
		where $V$ is a parabolic vector bundle over $C$ in the category. It is easy to see that the usual pullback 
		$\pi^*V_0$ is a subsheaf of the parabolic pullback $\pi_{\ppar}^*V$. Let $h: \pi^*V_0\to \pi_{\ppar}^*V$ be the natural inclusion of $G$-equivariant vector bundles. By Galois descent, we have a natural inclusion
		$$
		(\pi_*h)^G: V_0=(\pi_*\pi^*V_0)^G\to (\pi_{\ppar*}\pi_{\ppar}^*V)^G=(\pi_{\ppar*}\pi_{\ppar}^*V)_0.
		$$ 
		To verify $(\pi_*h)^G$ to be an isomorphism and respects the parabolic structure, it suffices to argue around each $D_i$. Hence the problem is reduced to the case of $V$ being a direct sum of parabolic line bundles. For the parabolic line bundle case, the statement is trivially true. The lemma is proved.    
	\end{proof}
	Next, we shall study the behavior of $\lambda$-connections under parabolic pushforward and parabolic pullback. Let $(E,\nabla)$ be a $\lambda$-connection over $C'_{\log}$. It is said to be $G$-equivariant, if $E$ is a $G$-equivariant vector bundle and the isomorphisms $\{\phi_g\}_{g\in G}$ are horizontal. We call the following result the Biswas-Iyer-Simpson correspondence for $\lambda$-connections. 
	\begin{proposition}\label{Biswas correspondence for Higgs/flat bundles in positive char}
		Notation as in \ref{parabolic_notation}. The parabolic pushfoward and parabolic pullback induce the equivalence of categories between the category of $G$-equivariant $\lambda$-connections over $C'_{\log}$ and the category of parabolic $\lambda$-connections over $C_{\log}$ whose parabolic structure is supported along $D$ and has weights in $\frac{1}{N}\Z$. Moreover, under this equivalence, $G$-equivariant $\lambda$-connections with nilpotent residues corresponds to adjusted parabolic $\lambda$-connections.
	\end{proposition}
	\begin{proof}
		Notice that a $\lambda$-connection $\nabla$ on a $G$-equivariant vector bundle $E$ makes a $G$-equivariant $\lambda$-connection if and only if the $\lambda$-connection
		$$
		\nabla: E\to E\otimes \omega_{C'_{\log}/k}
		$$
		is $G$-equivariant, where the right hand side is equipped with the tensor product of the given $G$-action on $E$ and the natural $G$-action on $\omega_{C'_{\log}/k}$. Therefore, when $D$ is empty, the proposition is nothing but the Galois descent. In general, the proposition holds away from $D$ for the same reason, and we we need to show is a statement locally around each $D_i$. However we should be aware that we \emph{cannot} assume in general that $(E,\nabla)$ is a direct sum of rank one $G$-equivariant $\lambda$-connections. We divide the whole proof into three small steps.  
		
		{\itshape Step 1. Induced $\lambda$-connections under parabolic pushfoward/pullback.} Let $(E,\nabla)$ be a $G$-equivariant $\lambda$-connection over $C'_{\log}$. Taking the usual pushforward and then the $G$-invariant, we obtain a $\lambda$-connection over $C_{\log}$:
		$$
		\pi_{\ppar*}\nabla: (\pi_{\ppar*}E)_0\to (\pi_{\ppar*}E)_0\otimes \omega_{C_{\log}/k}. 
		$$ 
		Around each $D_i$, we may write $E$ into a direct sum of rank one $G$-equivariant line bundles. Using that, it is easy to verify that $\pi_{\ppar*}\nabla$ respects the parabolic structure on $\pi_{\ppar*}E$. So $\pi_{\ppar*}\nabla: \pi_{\ppar*}E\to \pi_{\ppar*}E\otimes\omega_{C_{\log}/k}$ is a parabolic $\lambda$-connection. Conversely, let $(V,\nabla)$ be a parabolic $\lambda$-connection on $C_{\log}$. We equip $E_i^0=\pi^*V_0\otimes \sO_{C'}(ND')$ with the pullback $\lambda$-connection on the first factor and the canonical $\lambda$-connection on the second factor. Denote it by $\tilde \pi^*\nabla$. Because $\nabla$ respects the parabolic filtration on $V_0$, it follows that $E_i^j$s are $\tilde \pi^*\nabla$-invariant. So is the intersection $\pi_{\ppar}^*V$. The induced $\lambda$-connection on $\pi_{\ppar}^*V$ is denoted by $\pi_{\ppar}^*\nabla$. As $\tilde \pi^*\nabla$ is $G$-equivariant, so is $\pi_{\ppar}^*\nabla$. 
		
		{\itshape Step 2. Equivalence of categories for $\lambda$-connections.} Without $\lambda$-connection the equivalence is the statement of Lemma \ref{Biswas correspondence in positive char}. Away from $D$, this is the consequence of Galois descent. So the problem is local around $D$. Let $(E,\nabla)$ be a $G$-equivariant connection over $C'_{\log}$, and let $V=\pi_{\ppar*}E$ be the parabolic pushforward of $E$. Let $x$ be a local coordinate of $D_i\in C$ and $y$ be a local coordinate of $D'_i\in C'$ satisfying $\pi^*x=y^N$. Because
		$$
		E\subset \pi^*(V_0)\otimes \sO_{C'}(ND')=\pi^*(\pi_*(E)^G)\otimes \sO_{C'}(ND'),
		$$
		one may write any local section of $E$ around $D_i$ into the form $s\otimes \frac{f}{y^N}$ with $s$ a $G$-invariant section of $E$ and $f\in \sO_{C'}$. Therefore,
		$$
		\pi^*_{\ppar}\pi_{\ppar*}\nabla(s\otimes \frac{f}{y^N})=\lambda\cdot d(\frac{f}{y^N})s+\frac{f}{y^N}\pi_*\nabla(s)=\lambda\cdot d(\frac{f}{y^N})s+\frac{f}{y^N}\nabla(s)=\nabla(s\otimes \frac{f}{y^N}).
		$$ 
		So $\pi^*_{\ppar}\pi_{\ppar*}\nabla=\nabla$. We leave the verification of $\pi_{\ppar*}\pi_{\ppar}^*\nabla=\nabla$ for a parabolic $\lambda$-connection $\nabla$ over $C_{\log}$ as an exercise. 
		
		{\itshape Step 3. Adjustedness condition.} It remains to show that the nilpotence condition of a $\lambda$-connection over $C'_{\log}$ corresponds to the adjustedness condition of a $\lambda$-connection over $C_{\log}$. 
		
		Let $(E,\nabla)$ be a $G$-equivariant $\lambda$-connection over $C'_{\log}$ with nilpotent residues. We shall continue the local calculation started in Step 2. We may choose a local basis $\{e_{ij}\}_{1\leq i\leq w,1\leq j\leq s_i}$ of $E$ such that 
		$$
		\sigma(e_{ij})=\zeta^{-m_i}e_{ij},\quad \sigma(y)=\zeta y,
		$$
		and 
		$$
		0\leq m_1<m_2<\cdots<m_w\leq N-1.
		$$ 
		Then one computes that $V_0$ has a local basis $\{y^{m_{i}}e_{ij}\}$ and the parabolic filtration on $V_0$ (around $D_i$) is explicitly determined by
		$$
		V_{\frac{m_k}{N}}=(\pi_*(E\otimes \sO_{C'}(-m_kD')))^G=\bigoplus_{i,j} \sO_Xy^{m_i+n_i}e_{ij},\ 1\leq k\leq w.
		$$
		Set $e_i=\{e_{i1},\cdots,e_{is_i}\}_{1\leq i\leq w}$ and $e=\{e_1,\cdots,e_w\}$. Let $A'$ be the $\lambda$-connection one-form of $\nabla$ with respect to the local basis $e$, namely $\nabla(e)=eA'$. Write $A'=(A'_{ij})_{1\leq i\leq w,1\leq j\leq w}$ with $A'_{ij}$ a block matrix of log one-forms of size $s_i \times s_j$. Then, the $G$-equivariance property of $\nabla$ says that 
		\begin{equation}\label{G-equivariance property}
			A'_{ij}=y^{m_i-m_j}\pi^*A_{ij},
		\end{equation}
		where $A_{ij}$ is a block matrix of log one-forms on $C_{\log}$ with the same size as $A'_{ij}$. Fix the local basis $d\log x$ (resp. $d\log y$) of $\omega_{C_{\log}/k}$ (resp. $\omega_{C'_{\log}/k}$). They are related by 
		$$
		\pi^*d\log x=Nd\log y.
		$$
		So we may write uniquely $A_{ij}=a_{ij}d\log x$ (resp. $A'_{ij}=a'_{ij}d\log y$) with $a_{ij}$ (resp. $a'_{ij}$) a block matrix of local regular functions of $x$ (resp. $y$). Define $\mathrm{ord}_x(a_{ij})$ to be the minimum of the orders of $x$ among the entries of $a_{ij}$ and similarly define $\mathrm{ord}_y(a'_{ij})$. The relation \ref{G-equivariance property} yields the following inequality
		$$
		m_i-m_j+N\mathrm{ord}_x(a_{ij})\geq 0.
		$$
		It follows $\mathrm{ord}_x(a_{ij})>0$ when $i<j$ (which in turn implies $\mathrm{ord}_y(a'_{ij})>0$.). Notice $res(A'_{ij})=0$ for $i>j$ for trivial reason. Now that the residue of $\nabla$ at $D_i$ is nilpotent by condition, the eigenvalues of $res(A'_{ii})$ are zero for each $i$. With respect to the local basis $\{y^{m_1}e_1,\cdots,y^{m_w}e_w\}$ of $V_0$, the $\lambda$-connection one-form of $\pi_{\ppar*}\nabla$ takes the following form:
		$$B=\begin{pmatrix}
			A_{11}+\frac{\lambda \cdot m_1}{N}I_{s_1}d\log x & A_{12} &\cdots & A_{1w} \\
			A_{21} & A_{22}+\frac{\lambda \cdot  m_2}{N}I_{s_2}d\log x &\cdots & A_{2w}\\
			\vdots & \vdots & \ddots &\vdots\\
			A_{w1} & A_{w2}&\cdots & A_{ww}+\frac{\lambda \cdot  m_w}{N}I_{s_w}d\log x.
		\end{pmatrix}$$
		The residue matrix $res(B)$ is block lower triangular because the order of $x$ in any entry in $A_{ij}$ for $i<j$ is shown to be strictly positive. Since the eigenvalues of $res(A_{ii})$ are $N$-multiple of those of $res(A'_{ii})$ which are zero, it follows that the eigenvalues of $res(B)$ are exactly
		$\{\frac{\lambda \cdot m_i}{N}\}_{1\leq i\leq w}$. Summarizing the above discussions, we conclude that $\pi_{\ppar*}\nabla$ satisfies the adjustedness condition. 
		
		Conversely, let $(V,\nabla)$ be an adjusted parabolic $\lambda$-connection on $C_{\log}$. Let $E=\pi^*_{\ppar}V$ be the parabolic pullback of $V$. We are going to show the residues of $\pi_{\ppar}^*\nabla$ are nilpotent. Because of $V=\pi_{\ppar*}(\pi^*_{\ppar}V)=\pi_{\ppar*}E$,  the previous discussion provides us a set of local basis $\{y^{m_1}e_1,\cdots,y^{m_w}e_w\}$ of $V_0$ (around $D_i$), where $\{e_1,\cdots,e_w\}$ is a local basis of $E$ consisting of eigenvectors under the $G$-action. Now let $P=(P_{ij})$ be the $\lambda$-connection one-form of $\nabla$ on $V_0$ with respect to this basis. Here each $P_{ij}$ is a block matrix of one-forms. As $\nabla$ preservers the parabolic filtration, $res(P_{ij})=0$ for $i<j$. So $res(P)$ is a block lower triangular matrix. Now we let $Q=(Q_{ij})$ be the $\lambda$-connection one-form of $\pi_{\ppar}^*\nabla$ with respect to the basis $\{e_1,\cdots,e_w\}$. A simple calculation using $\lambda$-Leibniz rule yields the following equality 
		$$
		Q_{ij}=y^{m_i-m_j}\pi^*P_{ij}-\delta_{ij}\lambda\cdot m_iI_{s_i}d\log y. 
		$$
		Thus the matrix $res(Q)$ is block upper triangular (in fact block diagonal). By the adjustedness condition, the eigenvalues of $res(A_{ii})$ are exactly $\lambda \frac{m_i}{N}$ for each $i$. It follows that the eigenvalues of $res(Q)$ are
		$$
		\{-\lambda m_1+N\lambda \frac{m_i}{N}, \cdots, -\lambda m_w+N\lambda \frac{m_w}{N}\}
		$$ 
		which are all zeros. So $res(Q)$ is nilpotent as claimed. The proposition is proved. 
		
	\end{proof}
	
	\begin{remark}\label{exactness of parabolic pullback}
		The pair of the functors $(\pi^*_{\ppar}, \pi_{\ppar*})$ discussed above may be regarded as a twisted version of the classical Galois descent in finite \'etale case. Obviously, both functors are exact. That is, $\pi^*_{\ppar}$ sends a short exact sequence of parabolic $\lambda$-connections over $C_{\log}/k$ to a short exact sequence of $G$-equivariant $\lambda$-connections over $C'_{\log}/k$, and vice versa for $\pi_{\ppar*}$. Also, they commute with direct sum and tensor product. The proof is not difficult (see e.g. \cite[Lemma 4.3]{SW}).
	\end{remark}

Second, we observe that	there is a straightforward generalization of Theorem \ref{OV correspondence} for stacks as follows.
	\begin{notation}\label{log_root_curve_notation}
		
		Let $C/k$ be a smooth curve and $D\subset C$ be a reduced divisor, as in Notation \ref{basic notation}. Let $\{n_i\}_{1\leq i\leq l}$ be a set of positive integers coprime to $p$. Let $Z=C[\frac{D_1}{n_1},\cdots,\frac{D_l}{n_l}]$ be the (smooth, tame, Deligne-Mumford) root stack, with $\pi:Z\rightarrow C$ the map to the coarse moduli space. Let $Z_{\log}$ be the induced logarithmic root stack, equipped with the logarithmic structure determined by the divisor $\pi^*D\subset Z$. As $Z/k$ is smooth and tame, the sheaf of logarithmic K\"ahler differentials $\omega_{Z_{\log}/k}$ is a line bundle. Define the categories $\HIG^{lf}_{p-1}(Z_{\log}/k)$ and $\MIC^{lf}_{p-1}(Z_{\log}/k)$ as the corresponding categories for $C_{\log}$ in Notation \ref{basic notation}. As we are given the $W_2$-lifting $(\tilde C,\tilde D)$ of the log pair $(C,D)$, the log root stack admits a natural $W_2$-lifting $\tilde Z_{\log}$ too: let $\tilde Z=\tilde{C}[\frac{\tilde D_1}{n_1},\cdots,\frac{\tilde D_l}{n_l}]$ be the root stack over $W_2$ and $\tilde \pi: \tilde Z\to \tilde C$ the induced map to the coarse space. Equip $\tilde Z$ with the log structure determined by the divisor $\tilde \pi^*\tilde D$.
	\end{notation}
	By the standard descent argument, one deduces from Theorem \ref{OV correspondence} the following  
	\begin{proposition}\label{OV correspondence for tame root stack}
		Notation as in \ref{log_root_curve_notation}. Then there is an equivalence of categories
		$$
		\HIG^{lf}_{p-1}(Z_{\log}/k)\xrightleftharpoons[\ C_{Z_{\log}\subset \tilde Z_{\log}}\ ]{\ C_{Z_{\log}\subset \tilde Z_{\log}}^{-1}\ }\MIC^{lf}_{p-1}(Z_{\log}/k).
		$$
		When all $n_i$s equal one, the functors are the inverse Cartier/Cartier transform for $C_{\log}$ with respect to the lifting $\tilde C_{\log}$. 
	\end{proposition}
	\begin{proof}
		Take an open affine cover $C=\cup_iC_i$ positive integers $N_i$ coprime to $p$, finite cyclic groups $G_i$ of order prime to $p$, and $G_i$ covers $\pi_i\colon C'_i\rightarrow C_i$ branched over $D_i:=D\cap C_i$ such that that $Z_i:=Z\times_C C_i$ is isomorphic to the stack quotient $[C'_i/G_i]$. An object in $\HIG^{lf}_{p-1}(Z_{\log}/k)$ is obtained by gluing local objects in $\HIG^{lf}_{p-1}(Z_{i,\log}/k)$, and a similar statement is true for morphisms. Therefore we are reduced to proving are statement for each $C_i$ and ensuring that gluing maps are sent to gluing maps. On the other hand, for each $i$, there is an equivalence of categories of $\HIG^{lf}_{p-1}(Z_{i,\log}/k)$ and the category of $G$-equivariant Higgs bundles of $\HIG^{lf}_{p-1}(C'_{i,\log}/k)$. Since the inverse Cartier (with respect to the natural $W_2$-lifting $\tilde C_{i,\log}$) is an equivalence of categories, it \emph{automatically} promotes to an equivalence of the abelian categories of $G_i$-equivariant objects. The gluing is guaranteed because the inverse Cartier commutes with localization, i.e., is functorial with respect to open immersions of log schemes over $W_2$. This is the construction of the functor $C_{Z_{\log}\subset \tilde Z_{\log}}^{-1}$. The construction of the functor $C_{Z_{\log}\subset \tilde Z_{\log}}$ is done simply by reversing the direction of the above process. 
	\end{proof}
	
	Now we may proceed to the proof of Theorem \ref{ov correspondence for parabolic objects}.
	\begin{proof}[Proof of Theorem \ref{ov correspondence for parabolic objects}]

		Let $\pi: Z=C[\frac{D_1}{N},\cdots,\frac{D_l}{N}]\to C$ be the natural projection. Let $(E,\theta)$ be an object in $\HIG^{\ppar}_{p-1,N}(C_{\log}/k)$. We claim that $\pi^*$ induces the following equivalences of categories:
		\begin{itemize}
			\item [(i)] the category $\HIG_{p-1,N}^{\ppar}(C_{\log}/k)$ and the category $\HIG^{lf}_{p-1}(Z_{\log}/k)$
			\item [(ii)] the category $\mathrm{MIC}^{\ppar}_{p-1,N}(C_{\log}/k)$ and the category $\MIC^{lf}_{p-1}(Z_{\log}/k)$.   
		\end{itemize}
		Since $\theta$ is nilpotent of exponent $\leq p-1$ and $\pi$ is \'etale over the nonempty Zariski open subset $C-D$, $\pi^*\theta$ is nilpotent of exponent $\leq p-1$ over an nonempty Zariski open subset of $Z$ and hence over the whole stack $Z$. Moreover, applying the local computation of the parabolic pullback (Proposition \ref{Biswas correspondence for Higgs/flat bundles in positive char}) of a $\lambda$-connection (in our case $\lambda=0$) to each $\pi_i: C'_i\to C_i$, we get that $\pi^*(E,\theta)$ over each $C'_i$ is a $G$-equivariant Higgs bundle over $C'_{i,\log}$ (whose log structure is determined by the unique ramification point of $\pi_i$). Thus, over each $Z_i$, $\pi^*(E,\theta)$ is a Higgs bundle with pole along the unique ramification divisor. Conversely, by the local computation of the parabolic pushforward, a $G$-equivariant Higgs bundle over $C'_{i,\log}$ (equivalently a Higgs bundle over $Z_{i,\log}$) which is of nilpotent of exponent $\leq p-1$ must come from a parabolic Higgs bundle over $C_{i,\log}$ which is nilpotent of exponent $\leq p-1$. This shows the equivalence of categories (i). This argument extends to the flat objects. Thus Proposition \ref{Biswas correspondence for Higgs/flat bundles in positive char} implies the equivalence of categories (ii). Now we apply the inverse Cartier/Cartier transform in Proposition \ref{OV correspondence for tame root stack}, to complete the construction of $C^{-1}_{\ppar}$ and $C_{\ppar}$. 
	\end{proof}

For any real number $x$, its fractional part $\left\langle x\right\rangle $ is the unique real number satisfying 
$$
0\leq \left\langle x\right\rangle <1, \quad x\equiv \left\langle x\right\rangle  \textrm{modulo} \ \Z.
$$
Let $N$ be a positive integer and $p$ a prime with $(p,N)=1$. Choose an integer $\triangle$ satisfying $p\triangle \equiv 1\  \textrm{mod}\ N$. Let $(E,\theta)$ (resp. $(V,\nabla)$) be an object in $\HIG_{p-1,N}^{\ppar}(C_{\log}/k)$ (resp.  $\MIC_{p-1,N}^{\ppar}(C_{\log}/k)$). Let $D=\sum D_i$ and $\{\frac{m^1_{i}}{N},\cdots,\frac{m_i^{w_i}}{N}\}\subset [0,1)$ be the parabolic weights of $E$ (resp. $V$) at $D_i$. (Please note that the superscript $w_i$ is an index and does not refer to taking a power!) The following simple result reveals a basic fact about the parabolic Cartier/inverse Cartier transform.
\begin{proposition}\label{change of par weights under inverse Cartier}
Notation as above. The parabolic weights of $C^{-1}_{\ppar}(E,\theta)$ at $D_i$ are the following set 
$$
\{\left\langle \frac{pm_i^{1}}{N}\right\rangle,\cdots, \left\langle \frac{pm_i^{w_i}}{N}\right\rangle\}.
$$ 
Correspondingly, the parabolic weights of $C_{\ppar}(V,\nabla)$ at $D_i$ are given by
$$
\{\left\langle \frac{\triangle m_i^{1}}{N}\right\rangle,\cdots, \left\langle \frac{\triangle m_i^{w_i}}{N}\right\rangle\}.
$$ 
\end{proposition}
 \begin{proof}
Clearly, the statement for the Cartier transform follows that for the inverse Cartier transform, because $C_{\ppar}$ is a quasi-inverse to $C^{-1}_{\ppar}$ and 
$$
\frac{m}{N}=\left\langle \frac{p\triangle m}{N}\right\rangle,\quad 0\leq m\leq N-1.
$$
The problem is local around $D_i$.  So we may assume the following setting: Let $\pi: C'\to C$ be the cyclic cover of order $N$ such that $(E,\theta)$ is the parabolic pushforward of $(H,\eta)\in \HIG_{p-1}(C'_{\log}/k)$. $F$ admits a local basis $\{h_{ij}\}$ such that $\sigma(h_{ij})=\zeta^{-m_{ij}}h_{ij}$. See Proposition \ref{Biswas correspondence for Higgs/flat bundles in positive char} Step 3 for detail. By definition, 
$$
C^{-1}_{\ppar}(E,\theta)=\pi_{\ppar*}(C^{-1}(H,\eta)).
$$
Here is the key. Since we are in a local situation, the bundle part of $C^{-1}(H,\eta)$ is simply given by $F^*_{C}H$, where $F_C$ is the absolute Frobenius of $C$. Therefore, $F_C^*H$ admits a local basis $\{h_{ij}\otimes 1\}$ with the induced $G$-action given by
$$
\sigma(h_{ij}\otimes 1)=\zeta^{-pm_{ij}}(h_{ij}\otimes 1)=\zeta^{-m'_{ij}}(h_{ij}\otimes 1), \quad m'_{ij}=\left\langle \frac{pm_{ij}}{N}\right\rangle.
$$
The proposition follows.
 \end{proof}

	To conclude this section,  we record two lemmas for later use. Let $\pi: C'\to C$ be a cyclic cover of order $N$, which is coprime to $p$ and whose branch divisor is $D+B$ with $B$ a reduced divisor disjoint from $D$. Set $D'=\pi^{-1}D$ and $B'=\pi^{-1}B$. Then the parabolic pullback $\pi^*$ is a fully faithful functor from $\HIG^{\ppar}_{p-1,N}(C_{\log}/k)$ to $\HIG_{p-1}^{lf}(C'_{\log}/k)$, where $C'_{\log}=(C',D')$. Choose \emph{any} lifting $\tilde B\subset \tilde C$ of $B\subset C$. Then since $\pi: (C', D'+B')\to (C, D+B)$ is log \'etale, there is a unique $W_2$-lifting $(\tilde C', \tilde D'+\tilde B')$ of the pair $(C', D'+B')$ over $(C,D+B)$ up to isomorphism. 
	\begin{lemma}\label{parabolic inverse cartier commutes with pullback}
		Notation as above. Then for an object in $\HIG^{\ppar}_{p-1,N}(C_{\log}/k)$, one has the following natural isomorphism in $\MIC_{p-1}^{lf}(C'_{\log}/k)$:
		$$
		C^{-1}(\pi_{\ppar}^*(E,\theta))\cong \pi_{\ppar}^*(C^{-1}_{\ppar}(E,\theta)). 
		$$
	\end{lemma}
	
	We first state a sublemma.
	\begin{sublemma}\label{inverse Cartier commutes with finite etale}
		Let $\tilde C'_{\log}=(\tilde C',\tilde D')$ be a smooth curve over $W_2(k)$ equipped with a relative effective divisor and $\tilde \pi: \tilde C'_{\log}\to \tilde C_{\log}$ a \'etale morphism of log schemes over $W_2(k)$. Set $\pi: C'_{\log}\to C_{\log}$ to be the reduction of $\tilde \pi$ mod $p$. Then the inverse Cartier transform (with respect to those given $W_2$-liftings) commutes with the pull-back via $\pi$. More precisely, there exists a natural transformation of functors $\eta_{\pi}\colon C^{-1}\circ \pi^*\rightarrow \pi^*\circ C^{-1}$ that is an isomorphism. Moreover, if we have a third curve $\tilde{C}''_{\log}$ equipped with a \'etale morphism of log schemes $\tilde{\nu}\colon \tilde{C}''_{\log}\rightarrow \tilde{C}'_{\log}$, then the following diagram of functors and natural transformations commutes:
		\begin{equation}\label{compatibility between different pullbacks}\xymatrix{
				C^{-1}\circ \nu^*\circ \pi^*\ar[d]_{\eta_{\nu}}\ar[rr]^{\eta_{\pi\circ\nu}} & & \nu^*\circ \pi^*\circ C^{-1}\\
				\nu^*\circ C^{-1}\circ \pi^*.\ar[urr]_{\eta_{\pi}}  &  &
			}
		\end{equation}
	\end{sublemma}
	\begin{proof}
		This is a special case of \cite[Corollary 1.2]{Sh}.
	\end{proof}
	\begin{proof}[Proof of Lemma \ref{parabolic inverse cartier commutes with pullback}]
		Let $\eta: Z=C[\frac{D_1}{N},\cdots,\frac{D_l}{N}]\to C$ be the natural projection. Then there is a factorization of $\pi$
		$$
		\pi: C'\stackrel{\delta}{\to} Z\stackrel{\eta}{\to} C.
		$$
		The morphism $\delta$ is log \'etale. Therefore, it suffices to prove that $C^{-1}$ commutes with $\delta^*$. This amounts to \'etale descent in conjunction with Sublemma \ref{inverse Cartier commutes with finite etale}, and especially the compatibility guaranteed by \ref{compatibility between different pullbacks}, but for convenience we spell out the details.
		
		Let $Y\rightarrow Z$ be an (\'etale) atlas of $Z$. Then we have the following diagram:
		
		$$
		\xymatrix{
			C'\ar[d]^{\delta}&   C'_Y\ar[d]^{\delta_{Y}}\ar[l]&    C'_{Y\times_Z Y}\ar[d]^{\delta_{Y\times_Z Y}}\ar@<1ex>[l]^{q_1}\ar@<-1ex>[l]_{q_2}&   \ar@<1ex>[l]\ar[l]\ar@<-1ex>[l] C'_{Y\times_Z Y\times_Z Y}\ar[d]^{\delta_{Y\times_Z Y\times_Z Y}}\\
			Z& Y\ar[l]& Y\times_Z Y\ar@<1ex>[l]^{p_1}\ar@<-1ex>[l]_{p_2}& \ar@<1ex>[l]\ar[l]\ar@<-1ex>[l] Y\times_Z Y\times_Z Y
		}
		$$
		All of the vertical arrows are log \'etale. The category of Higgs bundles on the stack $Z$ is equivalent to the category of Higgs bundles with descent data on the nerve of the cover, i.e., the category whose objects are triples $(F,\theta,\varphi)$ where $(F,\theta)$ is a Higgs bundle on $Y$ and $\varphi\colon p_1^*(F,\theta)\rightarrow p_2^*(F,\theta)$ is an isomorphism of Higgs bundles on $Y\times_Z Y$ that satisfies the natural cocycle condition on $Y\times_Z Y\times_Z Y$. Similarly, the category of Higgs bundles on the scheme $C'$ is equivalent to the category of Higgs bundles on $C'_Y$ equipped with an isomorphism between the two pullbacks to $C'_{Y\times_Z Y}$ that satisfies the cocycle condition on $C'_{Y\times_Z Y\times_Z Y}$ by usual \'etale descent for coherent sheaves. There is a precisely analogous description of the category of modules with an integrable connection. 
		
		Under the above equivalences, the functor $\delta^*\circ C^{-1}$ may be described on objects via $(F,\theta,\varphi)\mapsto (\delta_Y^*(C^{-1}(F,\theta)),\delta_Y^*C^{-1}\varphi)$. Here, $C^{-1}\varphi$ is an isomorphism between $C^{-1}p_1^*(F,\theta)$ and $C^{-1}p_2^*(F,\theta)$, two modules with an integrable connection on $Y\times_Z Y$, that satisfies the cocycle condition. Using Sublemma \ref{inverse Cartier commutes with finite etale}, it follows that this may be thought of as an isomorphism $p_1^*C^{-1}(F,\theta)\rightarrow p_2^*C^{-1}(F,\theta)$. (We have implicitly suppressed the natural isomorphisms $\eta_{p_1}$ and $\eta_{p_2}$. This will be harmless by the compatibility guaranteed in \ref{compatibility between different pullbacks}.) Then $\delta_Y^*C^{-1}\varphi$ denotes the induced isomorphism:
		$$
		\delta_{Y\times_Z Y}^* p_1^* C^{-1}(F,\theta)\rightarrow \delta_{Y\times_Z Y}^*p_1^* C^{-1}(F,\theta).$$
		Again using Sublemma \ref{inverse Cartier commutes with finite etale}, we may think of this as an isomorphism $q_1^*\delta_Y^*C^{-1}(F,\theta)\rightarrow q_2^*\delta_Y^*C^{-1}(F,\theta)$ because $p_1\circ \delta_{Y\times_Z Y} =\delta_Y\circ q_1$ and $p_2\circ \delta_{Y\times_Z Y} =\delta_Y\circ q_2$. 
		
		Similarly, the functor $C^{-1}\circ \delta^*$ may be described on objects via $$(F,\theta,\varphi)\mapsto (C^{-1}\delta_Y^*(F,\theta)),C^{-1}\delta_Y^*\varphi).$$ Here, $C^{-1}\delta_Y^*\varphi$ is an isomorphism $p_1^*C^{-1}\delta_Y^*(F,\theta)\rightarrow p_2^*C^{-1}\delta_Y^*(F,\theta)$, where we have again used Sublemma \ref{inverse Cartier commutes with finite etale}.
		
		Now, Sublemma \ref{inverse Cartier commutes with finite etale} furnishes us with natural isomorphisms of functors $\eta_{\delta_Y}\colon C^{-1}\circ \delta_Y^*\rightarrow \delta_Y^*\circ C^{-1}$ and $\eta_{\delta_{Y\times_Z Y}}\colon C^{-1}\circ \delta_{Y\times_Z Y}^*\rightarrow \delta_{Y\times_Z Y}^*\circ C^{-1}$ that are compatible in the sense of Diagram \ref{compatibility between different pullbacks}. It follows that $\eta_{\delta_Y}$ will send the object $C^{-1}\delta_Y^*(F,\theta)$ to $\delta_Y^*C^{-1}(F,\theta)$ and moreover $\delta^*_YC^{-1}\varphi$ is sent to to $C^{-1}\delta_{Y}^*\varphi$. This defines a natural isomorphism $\eta_{\delta}\colon  C^{-1}\circ \delta^*\rightarrow \delta^*\circ C^{-1}$ by the above equivalence of categories of modules with an integrable connection on $C'$ and modules with an integrable connection on $C'_Y$ equipped with descent data.
	\end{proof} 
	\begin{lemma}\label{inverse Cartier commutes with tensor product}
		Let $(E_i,\theta_i),i=1,2$ be two objects in $\HIG^{\ppar}_{p-1,N}(C_{\log}/k)$. Let $l_i, i=1,2$ be the nilpotent index of $\theta_i$. If $l_1+l_2\leq p-1$, then the inverse Cartier transform commutes with tensor product, that is,
		$$
		C_{\ppar}^{-1}((E_1,\theta_1)\otimes (E_2,\theta_2))\cong C_{\ppar}^{-1}(E_1,\theta_1)\otimes C_{\ppar}^{-1}(E_2,\theta_2).
		$$
	\end{lemma}
	
	\begin{proof}
		Let $\pi: C'\to C$ be the cyclic cover in Lemma \ref{parabolic inverse cartier commutes with pullback}. As the parabolic pullback $\pi_{\ppar}^*$ is faithful, it suffices to show the above isomorphism after applying $\pi_{\ppar}^*$. Since $\pi_{\ppar}^*$ commutes with tensor product by Proposition \ref{Biswas correspondence for Higgs/flat bundles in positive char}, Lemma \ref{parabolic inverse cartier commutes with pullback} reduces the problem to the case of the trivial parabolic structure. We use the reinterpretation of the inverse Cartier transform via exponential twisting \cite{LSYZ}. Then by this case, the above isomorphism amounts to the fact that the following formal equality 
		$$
		\exp(\theta_1\otimes id+id\otimes \theta_2)=\exp(\theta_1\otimes id)\cdot \exp(id\otimes \theta_2)
		$$
		becomes a genuine equality under the above assumption on the nilpotent indices of $\theta_i$s.
	\end{proof}   
	
	As a matter of convention, we shall refer to $C^{-1}_{\ppar}$ (resp. $C_{\ppar}$) simply as $C^{-1}$ (resp. $C$). This shall not cause much confusion as they are natural extensions of Ogus-Vologodsky's functors to parabolic objects.

	\section{Parabolic de Rham bundles}
	Recall that a \emph{de Rham bundle} over $\sC_{\log}/S$ is a triple $(V,\nabla,Fil)$ where $(V,\nabla)$ is a flat bundle over $\sC_{\log}/S$ and $Fil$ is a finite decreasing, filtered free and Griffiths transverse filtration on $V$. Here \emph{filtered free} means that $(V,Fil)$ is Zariski locally isomorphic to a direct sum of $\sO_{\sC}[w_i]$s, where for $w\in \Z$, the underlying bundle $\sO_{\sC}[w]$ is $\sO_C$, and the filtration is given by:
	$$
	Fil^{w-i}=\sO_{\sC}, i\leq 0;\quad  Fil^{w+i}=0, i>0.
	$$
	On other other hand, a \emph{graded Higgs bundle} over $\sC_{\log}/S$ is a pair $(E,\theta)$, where $E=\bigoplus_iE^i$ is a direct sum of vector bundles, and $\theta$ is a Higgs field on $E$ satisfying
	$\theta(E^i)\subset E^{i-1}\otimes \omega_{\sC_{\log}/S}$. We shall generalize these two notions to the parabolic setting in the following. 
	
	\begin{definition}
		A \emph{parabolic de Rham bundle} over $\sC_{\log}/S$ is a triple $(V,\nabla,Fil_0)$ where $(V,\nabla)$ is a parabolic flat bundle over $\sC_{\log}/S$ and $(V_0,\nabla_0,Fil_0)$ is a de Rham bundle over $\sC_{\log}/S$. A morphism of parabolic de Rham bundles is a morphism of the underlying parabolic vector bundles which respects the connection and the filtration on $V_0$. A \emph{graded parabolic Higgs bundle} over $\sC_{\log}/S$ is triple $(E,\theta,\bigoplus_iE_0^i)$ where $(E,\theta)$ is a parabolic Higgs bundle and $(E_0=\bigoplus_iE_0^i,\theta_0)$ is a graded Higgs bundle over $\sC_{\log}/S$ satisfying the compatibility $E_{\alpha}=\bigoplus_i E_0^i\cap E_{\alpha}$ for any $\alpha\geq 0$. A morphism of graded parabolic Higgs bundles is a morphism of the underlying parabolic vector bundles which respects the Higgs field and the grading structure on $E_0$.  
	\end{definition}
	Let $\PaDR(\sC_{\log}/S)$ (resp. $\PaHG(\sC_{\log}/S)$) be the category of parabolic de Rham (resp. graded parabolic Higgs) bundles over $\sC_{\log}/S$.
	\begin{lemma}\label{the grading functor}
		There is an additive functor $\Gr: \PaDR(\sC_{\log}/S) \to \PaHG(\sC_{\log}/S)$, which sends a de Rham bundle $(V,\nabla,Fil)$, regarded as a parabolic de Rham bundle with trivial parabolic structure, to the parabolic graded Higgs bundle $(E=\Gr_{Fil}V,\theta=\Gr_{Fil}\nabla)$ with trivial parabolic structure. For a general parabolic de Rham bundle $(V,\nabla,Fil_0)$, the resulting $(E,\theta,\bigoplus_iE_0^i)=\Gr(V,\nabla,Fil_0)$ is called the \emph{associated} parabolic graded Higgs bundle. The weights of a parabolic de Rham bundle and its associated parabolic graded Higgs bundle coincide.
	\end{lemma}
	\begin{proof}
		Let $V$ be a parabolic vector bundle and let $Fil_0=Fil_0^{\bullet}$ be a finite decreasing and filtered free filtration on $V_0$. For each $i$ and each $n\in \N$, we define a filtration on $V_{n\delta_i}\subset V_0$ by
		$$
		Fil^{\bullet}_{n\delta_i}:=Fil^{\bullet}_0\cap V_{n\delta_i}.
		$$
		As $V_{0}=V_{n\delta_i}\otimes \sO_{\sC}(n\sD_i)$ and $Fil_0$ is filtered free, it follows that 
		$$
		Fil_0=Fil_{n\delta_i}\otimes Fil_{tr}
		$$
		where $Fil_{tr}$ stands for the trivial filtration (on $\sO_{\sC}(n\sD_i)$). This property allows us to equip each $V_{\alpha}$ a well-defined filtration from $Fil_0$: For $\alpha\geq 0$, we set the filtration on $V_{\alpha}$ by $Fil^{\bullet}_{\alpha}:=Fil_0^{\bullet}\cap V_{\alpha}$, and demand the filtration $Fil_{\alpha-\delta_i}$ on $V_{\alpha-\delta_i}=V_{\alpha}\otimes \sO_{\sC}(\sD_i)$ be the tensor product of $Fil_{\alpha}$ and the trivial filtration on $\sO_{\sC}(\sD_i)$. Now let $(V,\nabla,Fil_0)$ be a parabolic de Rham bundle. We shall verify that $Fil_{\alpha}$ on $V_{\alpha}$, constructed as above, satisfies Griffiths transversality. Because of the property $\nabla_{\alpha-\delta_i}=\nabla_{\alpha}\otimes id+id\otimes d$ (Claim \ref{tensor product}), it suffices to verify it for $\alpha\geq 0$. But this is clear:
		$$
		\nabla_{\alpha}(Fil^i_{\alpha})=\nabla_{\alpha}(Fil^i_0\cap V_{\alpha})\subset (Fil^{i-1}_0\cap V_{\alpha})\otimes \omega_{\sC_{\log}/S}=Fil^{i-1}_{\alpha}\otimes \omega_{\sC_{\log}/S}. 
		$$
		Thus we obtain the associated graded Higgs bundle for each $\alpha$:
		$$
		(E_{\alpha},\theta_{\alpha})=(\Gr_{Fil_{\alpha}}V_{\alpha},\Gr_{Fil_{\alpha}}\nabla_{\alpha}). 
		$$ 
		Clearly, the collection of $(E_{\alpha},\theta_{\alpha})$s and the natural inclusions $E_{\alpha}\hookrightarrow E_{\beta}$ for $\alpha\geq \beta$ form a parabolic Higgs bundle. For any $\alpha\geq 0$, one has a priori $\bigoplus_i E^i_0\cap E_{\alpha}\subset E_{\alpha}$. As 
		$$
		Fil^i_{\alpha}/Fil^{i+1}_{\alpha}=Fil^i_0\cap V_{\alpha}/Fil^{i+1}_0\cap V_{\alpha}\subset Fil^i_0/Fil^{i+1}_0,
		$$  
		it follows that 
		$$
		\bigoplus_i E^i_0\cap E_{\alpha}=\bigoplus_i (Fil^i_0/Fil^{i+1}_0\cap \bigoplus_i Fil^i_{\alpha}/Fil^{i+1}_{\alpha})\supset \bigoplus_i Fil^i_{\alpha}/Fil^{i+1}_{\alpha}=E_{\alpha},
		$$
		and hence $E_{\alpha}=\bigoplus_i E^i_0\cap E_{\alpha}$ as desired. The last statement on weights is clear by construction. 
	\end{proof}
	
	The essence of the above proof is the construction of $Fil_{\alpha}$ over $V_{\alpha}$ from the filtration $Fil_0$ over $V_{0}$, such that $Fil=\{Fil_{\alpha}\}_{\alpha}$ is a filtration of $V=\{V_{\alpha}\}_{\alpha}$ by parabolic subbundles.  So for a parabolic de Rham bundle $(V,\nabla,Fil_0)$, we may equally call $Fil$ as above \emph{the Hodge filtration} for $(V,\nabla)$.

	An important class of parabolic de Rham bundles arises from geometry. Let $C$ be a smooth projective complex curve. Consider a smooth projective morphism $f: X\to U$ of relative dimension $d\geq 0$, where $U\subset C$ is a dense Zariski open subset. Resolution of singularities allows us to compactify $f$ into a \emph{quasi-semistable} family over $C$. That is, we have a projective morphism $\bar f\colon \bar X\to C$ whose restriction to $U$ is isomorphic to $f$, and such that \'{e}tale locally $\bar f$ is the product of morphisms of following type:
	\begin{itemize}
		\item[(i)] $pr: \A^n\to \A^1$;
		\item[(ii)] $g: \A^n=\Spec \C[y_1,\cdots,y_n]\to \A^1=\Spec \C[x]$ with $g^*x=y_1^{e_1}\cdots y_n^{e_n}$ for a non-zero tuple of non-negative integers $(e_1,\cdots,e_n)$.
	\end{itemize}
	Set $D=C-U$ and $E=\bar f^{-1}(D)$ and $\bar H^i_{dR}=\mathbb{R}^i\bar f_*\Omega^{\cdot}_{\bar X/C}(\log E/D)$ for some $0\leq i\leq 2d$. Steenbrink \cite{St76} shows that over $\overline H^i_{dR}$, there is a logarithmic connection $\overline \nabla$ whose residues have eigenvalues in $[0,1)\cap \Q$, and which actually coincides with Deligne's canonical extension \cite{Del70} of the Gau{\ss}-Manin connection $\nabla$ on $H^i_{dR}(X/U)$. Moreover, he shows that the Hodge spectral sequence of the relative logarithmic de Rham complex degenerates at $E_1$, so that the Hodge filtration $Fil$ on $H^i_{dR}(X/U)$ extends to a Hodge filtration $\overline{Fil}$ on $\overline H^i_{dR}$. They forms the $i$-th \emph{logarithmic Gau{\ss}-Manin system} $\overline H_{dR}:=(\overline H^i_{dR},\overline \nabla,\overline{Fil})$ associated to $\bar f$, which is a de Rham bundle over $C_{\log}$.  
	
	We have the following extension lemma which is an enhanced version of \cite[Lemma 3.3]{IS} in the curve case. 
	\begin{lemma}\label{canonical parabolic extension}
		Let $C$ be a smooth projective complex curve and $D\subset C$ a reduced effective divisor. Then Deligne extension induces an equivalence of categories between the category of de Rham bundles over $U$ with regular singularities and rational residues and the category of adjusted parabolic de Rham bundles over $C_{\log}$. It is compatible with exact sequence, direct sums, tensor products and internal hom. A quasi-inverse of the Deligne extension is given by the restriction to $U$.  
	\end{lemma}
	For a de Rham bundle $V$ over $U$ in the above lemma, we shall call the associated adjusted parabolic de Rham bundle over $C_{\log}$ the \emph{canonical parabolic extension} of $V$. 
	\begin{proof}
		When the filtration is trivial, the lemma is nothing but \cite[Lemma 3.3, Lemma 3.6]{IS}.  For a general de Rham bundle $(V,\nabla,Fil)$ over $U$ in the category, one associates it by \cite[Lemma 3.3]{IS} an adjusted parabolic flat bundle $(\overline V, \overline \nabla)$ by forgetting $Fil$. Let $\xi$ be the generic point of $C$. Applying \cite[Proposition 1]{Langton} to 
		$$
		Fil^{\bullet}_{\xi}=Fil^{\bullet}\otimes_{\sO_U}\sO_{\xi,U}\subset V\otimes_{\sO_U}\sO_{\xi,U}=\overline V\otimes _{\sO_C}\sO_{\xi,C},
		$$
		one obtains a unique filtration $\overline{Fil}=\overline{Fil}^{\bullet}$ on $\overline V$ by \emph{subbundles}  which extends $Fil$ on $V$. It satisfies Griffiths transversality with respect to $\overline \nabla$ because the transversality condition holds over the generic point. Therefore, the triple $(\overline V, \overline \nabla,\overline{Fil})$ is indeed an adjusted parabolic de Rham bundle. Clearly, it is an equivalence of categories. We also note that, varying the nonempty subset $U\subset C$, one obtains a compatible family of extension functors.  
	\end{proof}
	Applying Lemma \ref{canonical parabolic extension} to the Gau{\ss}-Manin connection $(H^i_{dR}(X/U),\nabla)$, one obtains the \emph{canonical} parabolic structure on $\overline H_{dR}$ which is preserved by $\overline \nabla$. By \cite[Remark 3.4]{IS}, the canonical parabolic structure is trivial if and only if the residues of $\overline \nabla$ are nilpotent. Also by Lemma \ref{canonical parabolic extension}, the extended filtration $\overline{Fil}$ is independent of the choice of a quasi-semistable compactification of $f$. Therefore, we shall call the so-obtained adjusted parabolic de Rham bundle $\overline H_{dR}:=(\overline H^i_{dR},\overline \nabla,\overline{Fil})$ the $i$-th degree \emph{parabolic Gau{\ss}-Manin system} associated to $f$. 
	\begin{definition}\label{motivic objects}
		Let $C$ be a smooth projective curve over $\C$.  An adjusted parabolic de Rham bundle $(V,\nabla,Fil)$ over $C_{\log}$ is said to be \emph{motivic}, if over the generic point $\xi\in C$, it is isomorphic to a subspace of the restriction of a Gau{\ss}-Manin system over some nonempty Zariski open subset of $C$ to $\xi$. A parabolic graded Higgs bundle over $C_{\log}$ is said to be \emph{motivic} if it is the associated parabolic graded Higgs bundle to a motivic parabolic de Rham bundle over $C_{\log}$. 
	\end{definition}
	Now we resume to the setting of Notation \ref{parabolic_notation}. Consider a parabolic de Rham bundle $(V,\nabla,Fil_0)$ over $C_{\log}/k$, such that $(V,\nabla)$ is a parabolic flat bundle whose weights are contained in $\frac{1}{N}\Z$. Let $Fil$ be the Hodge filtration for $(V,\nabla)$. By the BIS correspondence, $\pi_{\ppar}^*Fil$ is a filtration of $\pi_{\ppar}^*V$ by subbundles.  
	
	\begin{proposition}\label{parabolic pullback commutes with grading functor}
		The parabolic pullback functor commutes with the grading functor $\Gr$ in Lemma \ref{the grading functor}. More precisely, let $(V,\nabla,Fil_0)$ be a parabolic de Rham bundle over $C_{\log}$ as above. Then there is a natural isomorphism of graded Higgs bundles over $C'_{\log}$:
		$$
		Gr_{\pi_{\ppar}^*Fil}(\pi_{\ppar}^*V,\pi_{\ppar}^*\nabla)\cong \pi_{\ppar}^*Gr_{Fil}(V,\nabla).
		$$
	\end{proposition}
	\begin{proof}
	It follows from the exactness property of parabolic pullback (take a $\lambda$-connection in Remark \ref{exactness of parabolic pullback} to be the zero Higgs field), that there is a natural isomorphism of ($G$-equivariant) vector bundles:
	$$
	Gr_{\pi_{\ppar}^*Fil}\pi_{\ppar}^*V\cong \pi_{\ppar}^*Gr_{Fil}V.
	$$
	To show the isomorphism respects the Higgs structure, it suffices to considering its restriction to the generic point $\xi_{C'}$ of $C'$.  Note that there is another way to characterize $\pi_{\ppar}^*Fil$: Indeed, again by \cite[Proposition 1]{Langton}, it is the unique filtration on $\pi_{\ppar}^*V$  by subbundles, which over $\xi_{C'}$ equals the usual pullback of $Fil_{\xi_C}$. Since $\pi_{\ppar}^*$ over $\xi_C$ is just the usual pullback, it follows that the restrictions of both $Gr_{\pi_{\ppar}^*Fil}\pi_{\ppar}^*\nabla$ and $\pi_{\ppar}^*Gr_{Fil}\nabla$ are naturally isomorphic to the usual pullback of the graded Higgs field $Gr_{Fil}\nabla$. 		
	\end{proof}

	\section{Periodic de Rham bundles}\label{section:periodic}
	The aim of the section is to introduce the notion of periodic parabolic Higgs/de Rham bundles over a smooth complex curve. Using Theorem \ref{ov correspondence for parabolic objects} and Lemma \ref{the grading functor}, we may generalize the notion of a \emph{periodic Higgs-de Rham flow} \cite[Defintion 3.1]{LSZ} to the parabolic setting.
	\begin{definition}\label{def of PHDF}
		
		Notation as in \ref{basic notation}. (In particular, we are working in characteristic $p>0$.) A \emph{periodic parabolic Higgs-de Rham flow} over $C_{\log}$ of period $f\in \N_{>0}$ is a diagram as follows:
		
		$$
		\xymatrix{
			& C^{-1}(E_0,\theta_0)\ar[dr]^{Gr_{Fil_0}} && C^{-1}(E_{f-1},\theta_{f-1})\ar[dr]^{Gr_{Fil_{f-1}}} \\
			(E_0,\theta_0) \ar[ur]^{C^{-1} } & & \cdots \ar[ur]^{C^{-1}}&& (E_f,\theta_f),\ar@/^2pc/[llll]^{\stackrel{\phi}{\cong} } }
		$$
		
		where
		
		\begin{itemize}
			
			\item the initial term $(E_0,\theta_0)$ is a parabolic graded Higgs bundle lying in the category $\HIG^{\ppar}_{p-1,N}(C_{\log}/k)$ by ignoring the grading structure,
			
			\item for each $i\geq 0$ $Fil_i$ is a Hodge filtration on the adjusted parabolic flat bundle $C^{-1}(E_i,\theta_i)$ making it into a adjusted parabolic de Rham bundle of level $\leq p-1$,
			
			\item $(E_i,\theta_i),i\geq 1$ is the associated parabolic graded Higgs bundle to  $$(C^{-1}(E_{i-1},\theta_{i-1}),Fil_{i-1}),$$
			
			\item and $\phi$ is an isomorphism of parabolic graded Higgs bundles.
			
		\end{itemize}
		We write $((E_0,\theta_0),Fil_0,\cdots,Fil_{f-1},\phi)$ for an $f$-periodic Higgs-de Rham flow. A morphism from $((E_0,\theta_0),Fil_0,\cdots,Fil_{f-1},\phi)$ to another $f$-periodic Higgs-de Rham flow $((E'_0,\theta'_0),Fil'_0,\cdots,Fil'_{f-1},\phi')$ is a morphism 
		$$
		g: (E_0,\theta_0)\to (E'_0,\theta'_0)
		$$
		of parabolic graded Higgs bundles, such that the induced morphisms of parabolic flat bundles by $C^{-1}$ respects the filtrations consecutively and such that the following diagram commutes:
		\begin{equation*}\label{eq1}
			\begin{CD}
				Gr_{Fil_{f-1}}C^{-1}(E_{f-1},\theta_{f-1})@>\phi>>(E_0,\theta_0)\\
				@V(\Gr\circ C^{-1})^{f}(g)VV@  VVgV\\
				Gr_{Fil'_{f-1}}C^{-1}(E'_{f-1},\theta'_{f-1})@>\phi'>>(E'_0,\theta'_0).
			\end{CD}
		\end{equation*}
	\end{definition}
	They form the category $\HDF^{\ppar}_{p-1,f}(\tilde C_{\log}/W_2)$ (which contains $\HDF_{p-1,f}(\tilde C_{\log}/W_2)$ of \S3 \cite{LSZ} as full subcategory). In a completely symmetric manner, one defines the notion of a \emph{periodic de Rham-Higgs flow} of period $f$. Again in terms of a diagram, a periodic de Rham-Higgs flow of period $f$ with the initial term an adjusted parabolic de Rham bundle $(V_0,\nabla_0,Fil_0)$ looks like as follows:
	$$\xymatrix{
		(V_0,\nabla_0)\ar[dr]^{Gr_{Fil_0}} & & \cdots\ar[dr]^{Gr_{Fil_{f-1}}} & & (V_f,\nabla_f)\ar@/_2pc/[llll]_{\stackrel{\psi}{\cong} } \\
		& Gr_{Fil_0}(V_0,\nabla_0)\ar[ur]^{C^{-1}} & &Gr_{Fil_{f-1}}(V_{f-1},\nabla_{f-1}).\ar[ur]^{C^{-1}} &
	}
	$$
	We denote $\DHF^{\ppar}_{p-1,f}(\tilde C_{\log}/W_2)$ for the category of $f$-periodic parabolic de Rham-Higgs flows. When we speak of a periodic flow, it means either a periodic parabolic Higgs-de Rham flow or a periodic parabolic de Rham-Higgs flow. Using the first term of a periodic flow, one obtains the natural forgetful functors:
	$$
	\DHF^{\ppar}_{p-1,f}(\tilde C_{\log}/W_2)\to \PaDR(C_{\log}/k), \quad \HDF^{\ppar}_{p-1,f}(\tilde C_{\log}/W_2)\to \PaHG(C_{\log}/k).
	$$
	We denote by $\PPDR_{f}(\tilde C_{\log}/k)$ the subcategory of $\PaDR(C_{\log}/k)$ which is the image of the category $\DHF^{\ppar}_{p-1,f}(\tilde C_{\log}/W_2)$. Note that for two positive integers $f,f'$, we can compose a morphism in $\PPDR_{f}(\tilde C_{\log}/k)$ with a morphism in $\PPDR_{f'}(\tilde C_{\log}/k)$ in the category $\PPDR_{\mathrm{lcm}(f,f')}(\tilde C_{\log}/k)$. Thus, the union of objects and morphisms in the categories $\PPDR_{f}(\tilde C_{\log}/k)$, where $f$ runs through all positive integers, forms a subcategory of $\PaDR(C_{\log}/k)$, which is denoted by $\PPDR(\tilde C_{\log}/k)$.

	An object in $\PPDR_{f}(\tilde C_{\log}/k)$ is called a \emph{periodic parabolic de Rham bundle} of period $f$, and an object in $\PPDR(\tilde C_{\log}/k)$ is  simply called a \emph{periodic parabolic de Rham bundle} (with respect to the $W_2$-lifting $\tilde C_{\log}$). Note that a morphism of two periodic parabolic de Rham bundles in the category $\PPDR(\tilde C_{\log}/k)$ simply means a morphism of the underlying parabolic de Rham bundles with the following property: there exist periodic de Rham-Higgs flows with initial de Rham terms the given ones such that the morphism of the initial parabolic de Rham bundle extends to a morphism of the periodic flows. In an analogous manner,  one defines the subcategories $\PPHG_{f}(\tilde C_{\log}/k)$ and $\PPHG(\tilde C_{\log}/k)$ of $\PaHG(C_{\log}/k)$. The latter is called the category of \emph{periodic parabolic Higgs bundles}. The two categories $\PPDR(\tilde C_{\log}/k)$ and $\PPHG(\tilde C_{\log}/k)$ of periodic objects contain obvious full subcategories whose objects have trivial parabolic structure. These are important subcategories-by the cyclic covering trick, we can always pulls back (via parabolic pullback) a periodic parabolic object to a one with trivial parabolic structure (see Lemma \ref{periodic parabolic pulls back to parabolic} below). We denote them by $\LPDR(\tilde C_{\log}/k)$ and $\LPHG(\tilde C_{\log}/k)$ respectively. Their objects are called periodic de Rham bundles and periodic Higgs bundles respectively. We shall list two basic properties about the categories of periodic objects. 
	\begin{lemma}\label{abelian category}
		Notations as above. The following statements hold:
		\begin{itemize}
			\item [(i)] The grading functor 
			$$
			\Gr: \PPDR(\tilde C_{\log}/k)\to \PPHG(\tilde C_{\log}/k)  
			$$
			is additive, faithful and essentially surjective.
			\item [(ii)] The full subcategories $\LPDR(\tilde C_{\log}/k)$ and $\LPHG(\tilde C_{\log}/k)$ are abelian.
		\end{itemize}
	\end{lemma}
	We remark that the categories $\PPDR(\tilde C_{\log}/k)$ and $\PPHG(\tilde C_{\log}/k)$ are closed under direct sum (but not necessarily under tensor product); the abelian property about subcategories in (ii) is surprising - it implies morphisms of periodic de Rham bundles are automatically \emph{strict} for the filtration.  We do not know the whole categories are abelian or not, and this seems to be an interesting question.
	\begin{proof}
		In a periodic parabolic Higgs-de Rham flow, the associated graded of the last de Rham term is, by definition, isomorphic to the initial Higgs term. This means the functor $\Gr$ is essentially surjective. Let $V, V'$ be two two periodic parabolic de Rham bundles. We need to show the map 
		$$
		\Gr: \Hom(V,V')\to \Hom(\Gr V, \Gr V')
		$$
		is injective. By definition, a morphism from $V$ to $V'$ extends to a morphism of periodic flows. So if its grading is zero, then all morphisms after it in the flow are zero too. By periodicity, the starting morphism has to be zero. 
		
		For the abelian property of the subcategories, it suffices to show that every morphism between periodic de Rham/Higgs bundles has a kernel and a cokernel. Let $g: E\to E'$ be a morphism of periodic de Rham bundles. By \cite[Theorem 1.1]{LSYZ}, the category $\HDF_{f,p-1}(\tilde C_{\log}/W_2)$ is naturally equivalent to the category $\mathrm{MF}^{\nabla}_{[0,p-1],f}(\tilde C_{\log}/W_2)$ of strict $p$-torsion logarithmic Fontaine modules with endomorphism structure $\F_{p^f}$. Then, it follows from \cite[Theorem 2.1]{Fa88} that the latter category is abelian. Therefore, $\ker(g)$ as well as $\mathrm{coker}(g)$ are again periodic. For its de Rham counterpart, it suffices to notice that the category  $\DHF_{f,p-1}(\tilde C_{\log}/W_2)$ is equivalent to $\HDF_{f,p-1}(\tilde C_{\log}/W_2)$. The lemma is proved.
	\end{proof}
	\begin{lemma}\label{periodic parabolic pulls back to parabolic}
		Let $(C,D)$ and $(\tilde C, \tilde D)$ be as above. Then there is a cyclic covering $\pi:C'\to C$ whose ramification divisor contains $D$ such that the parabolic pullback of any object of $\PPDR(\tilde C_{\log}/k)$ (resp. $\PPHG(\tilde C_{\log}/k)$) lies in $\LPDR(\tilde C'_{\log}/k)$ (resp. $\LPHG(\tilde C'_{\log}/k)$). 
	\end{lemma}
	\begin{proof}
		Consider first a projective $C$. We may always choose an effective divisor $Z$ disjoint from $D$ such that the degree $D+Z$ is divisible by $N$. Hence we may take an $N^{\text{th}}$ root of $D+Z$ globally, that is, there is a cyclic cover $\pi: C'\to C$ of order $N$ which totally ramifies along $D+Z$. If $C$ is affine, we may even have a cyclic cover of order $N$ whose ramification locus is exactly $D$: let $\bar C$ be the compactification. Take a point $P\in \bar C-C$. Then there is a natural number $d$ such that the divisor $D+dP\subset \bar C$ is divisible by $N$. Therefore, we get a cyclic cover $\bar \pi: \bar C'\to \bar C$ of order $N$ which ramifies along $D+P$. Restricting $\bar \pi$ to the inverse image of $C$, we obtain the cyclic cover as desired. In summary, we may always find a cyclic cover $\pi: C'\to C$ of order $N$ whose ramification divisor is of form $D+Z$ with $Z$ reduced and disjoint from $D$. Set $D'=\pi^{-1}D$ and $B'=\pi^{-1}B$. So we are in the situation of Lemma \ref{parabolic inverse cartier commutes with pullback}. Now let $(V,\nabla,Fil)$ be a periodic parabolic de Rham bundle which initializes a periodic parabolic de Rham-Higgs flow over $(C,D)$ of period $f$. Let $(\pi_{\ppar}^*V,\pi_{\ppar}^*\nabla, \pi_{\ppar}^*Fil)$ be the parabolic pullback of $(V,\nabla,Fil)$. Proposition \ref{parabolic pullback commutes with grading functor} provides an isomorphism of graded Higgs bundles over $(C',D')$:
		$$
		Gr_{\pi_{\ppar}^*Fil}(\pi_{\ppar}^*V,\pi_{\ppar}^*\nabla)\cong \pi_{\ppar}^*(Gr_{Fil}(V,\nabla)).
		$$ 
		Next we apply Lemma \ref{parabolic inverse cartier commutes with pullback} and obtain the following isomorphism of parabolic flat bundles:
		$$
		C^{-1}\circ Gr_{\pi_{\ppar}^*Fil}(\pi_{\ppar}^*V,\pi_{\ppar}^*\nabla)\cong \pi_{\ppar}^*(C^{-1}\circ Gr_{Fil}(V,\nabla)).
		$$
		Finally, by Proposition \ref{Biswas correspondence for Higgs/flat bundles in positive char}, the isomorphism $$\psi: (V_f,\nabla_f)\cong (V_0,\nabla_0)=(V,\nabla)$$ yields an isomorphism of flat bundles via the parabolic pullback
		$$
		\pi_{\ppar}^*\psi: \pi_{\ppar}^*(V_f,\nabla_f)\cong \pi_{\ppar}^*(V,\nabla).
		$$
		Adding the previous discussions altogether, we have shown that the parabolic pullback of a periodic parabolic de Rham-Higgs flow over $(C,D)$ is a periodic de Rham-Higgs flow over $(C',D')$ (of the same period). The argument works verbatim for periodic parabolic Higgs-de Rham flows. The lemma follows. 
	\end{proof}
	Recall that for a smooth projective $C$ over an algebraically closed field and a parabolic vector bundle $V$ over it, the \emph{parabolic degree} of $V$ is defined as follows:
	$$
	\pdeg(V):=\deg V_0+\sum_{i=1}^l\sum_{j=0}^{w_i}\alpha_i^j\dim Gr_{D_i,\alpha_i^j}V.
	$$
	The definition implies that $\deg(\pi_{\ppar}^*V)=N\pdeg(V)$ holds for any parabolic vector bundle $V$ whose parabolic structure is supported in $D$ and weights in $\frac{1}{N}\Z$ and $\pi: C'\to C$ as given in Notation \ref{parabolic_notation}.
	\begin{proposition}\label{zero par degree in char p}
		Notation as in \ref{basic notation} and suppose that $C$ is projective. Then a periodic parabolic Higgs bundle over $C_{\log}/k$ is parabolic semistable of degree zero. 
	\end{proposition}
	\begin{proof}
		Let $(E,\theta)$ be a periodic parabolic Higgs bundle. By Lemma \ref{periodic parabolic pulls back to parabolic}, there is a degree $N$ cyclic covering $\pi: C'\to C$ such that $\pi_{\ppar}^*(E,\theta)$ is a periodic Higgs bundle. By \cite[Proposition 6.3]{LSZ}, $\deg \pi_{\ppar}^*E=0$. Then $\pdeg(E)=\frac{1}{N}\deg \pi_{\ppar}^*E=0$. By Proposition \ref{Biswas correspondence for Higgs/flat bundles in positive char}, parabolic Higgs subbundles corresponds to $G$-equivariant Higgs subbundles of $\pi_{\ppar}^*(E,\theta)$. By \cite[Proposition 6.3]{LSZ} again, $\pi_{\ppar}^*(E,\theta)$ is semistable and hence any Higgs subsheaf has degree $\leq 0$. It follows that any parabolic Higgs subsheaf of $(E,\theta)$ has parabolic degree $\leq 0$.  
	\end{proof}
The following consequence is immediate.
\begin{corollary}\label{parabolic degree of periodic parabolic dR bundle}
	Notation as above. A periodic parabolic de Rham bundle over $C_{\log}/k$ is of parabolic degree zero.
\end{corollary}	
The next result is useful in applications.
\begin{proposition}\label{black hole principle}
	Notations as Proposition \ref{zero par degree in char p}. If the initial term of a periodic parabolic Higgs-de Rham flow over $C_{\log}$ is parabolic stable, then any intermediate Higgs term of the periodic flow is also parabolic stable.
\end{proposition}
\begin{proof}
The proof of \cite[Lemma 2.4]{LS} works verbatim in the parabolic setting. 
\end{proof}

	We come to the definition of periodic parabolic de Rham bundles and periodic parabolic Higgs bundles over $\C$.
	
	\begin{definition}\label{arithmetic_definition}
		Let $(C,D)$ be a logarithmic curve over $\C$.
		\begin{enumerate}
			\item A parabolic de Rham bundle $(V,\nabla,Fil)$ over $C_{\log}$ is called \emph{periodic} if it is adjusted and there exists:
			\begin{itemize}
				\item an integer $f$;
				\item a spreading-out $$(\mathscr{C},\scrD, \scrV,\nabla,{\mathscr Fil})\rightarrow S,$$ where $S$ is an integral scheme of finite type over $\Z$; and
				\item a proper closed subscheme $Z\subset S$,  
			\end{itemize} 
			such that for all geometric points $s\in S-Z$, the reduction $(\scrV_s,\nabla_s, {\mathscr Fil}_s)$ at $s$ belongs to $\PPDR_f(\mathscr{C}_{\tilde s,\log}/k(s))$ for \emph{all} $W_2(k(s))$-lifting $\tilde s\to S-Z$.  A morphism $\alpha: (V_1,\nabla_1,Fil_1)\to (V_2,\nabla_2,Fil_2)$ of periodic parabolic de Rham bundles is a morphism of parabolic de Rham bundles such that the reduction $\alpha_s$ at $s$ belongs to $\PPDR_f(\mathscr{C}_{\tilde s,\log}/k(s))$ for all $W_2(k(s))$-lifting $\tilde s\to S$.
			
			\item A parabolic Higgs bundle $(E,\theta)$ over $C_{\log}$ is called \emph{periodic} if it is graded and there exists:
			\begin{itemize}
				\item an integer $f$;
				\item a spreading-out $$(\mathscr{C},\scrD, \scrE,\Theta)\rightarrow S,$$ where $S$ is an integral scheme of finite type over $\Z$; and
				\item a proper closed subscheme $Z\subset S$,  
			\end{itemize} 
			such that for all geometric points $s\in S-Z$, the reduction $(\scrE_s,\Theta_s)$ at $s$ belongs to $\PPHG_f(\mathscr{C}_{\tilde s,\log}/k(s))$ for \emph{all} $W_2(k(s))$-lifting $\tilde s\to S-Z$. A morphism $\beta: (E_1,\theta_1)\to (E_2,\theta_2)$ of periodic parabolic Higgs bundles is a morphism of parabolic graded Higgs bundles such that the reduction $\beta_s$ at $s$ belongs to $\PPHG_f(\mathscr{C}_{\tilde s,\log}/k(s))$ for all $W_2(k(s))$-lifting $\tilde s\to S$. 
		\end{enumerate}
		
	\end{definition}
	One notes that the above definition for $D=\emptyset$ is nothing but Definition \ref{definition of motivic and periodic objects}. The next lemma allows one to speak of \emph{periodicity} for an object of either $\PaDR(C_{\log}/\C)$ or $\PaHG(C_{\log}/\C)$ over the field of complex numbers. 
	\begin{lemma}\label{independence of integral model}
		If a parabolic de Rham/Higgs bundle over $C_{\log}$ is periodic with respect to one spreading-out, then it is periodic for any spreading-out.
	\end{lemma}
	\begin{proof}
		We shall prove the lemma \emph{only} for a parabolic Higgs bundle $(E,\theta)$; the proof for a parabolic de Rham bundle is the same. Let $(\sC_1, \sD_1,\sE_1,\Theta_1)\to S_1=\Spec(A_1)$ be a spreading-out of $(C,D,E,\theta)$ satisfying the periodicity condition in Definition \ref{arithmetic_definition}. By shrinking $S_1$, we may assume that the proper closed subscheme $Z_1\subset S_1$ is empty. Now let $(\sC_2, \sD_2,\sE_2,\Theta_2)\to S_2=\Spec(A_2)$ be another spreading-out of $(C,D,E,\theta)$. Clearly, there exists some finitely generated $\Z$-subalgebra $A\subset \C$ satisfying the following properties:
		\begin{itemize}
			\item[(a)] $A_i\subset A, i=1,2$;
			\item[(b)] The isomorphisms $(\sC_i,\sD_i,\sE_i,\Theta_i)\otimes_{A_i} \C \cong (C,D,E,\theta), i=1,2$ are defined over $A$; 
			\item[(c)] The natural morphisms $S:=\Spec(A)\to S_i, i=1,2$ are smooth. 
		\end{itemize}
		Set $(\sC,\sD,\sE,\Theta)=(\sC_1,\sD_1,\sE_1,\Theta_i)\otimes_{A_1} A$; by definition of $A$, this is isomorphic to $(\sC_2,\sD_2,\sE_2,\Theta_2)\otimes_{A_2}A$ over $A$ and is isomorphic to $(C,D,E,\theta)$ over $\C$. Let $\pi_1: (\sC,\sD)\to (\sC_1,\sD_1)$ be the natural morphism of log schemes over $S$. Take any $\tilde s\in S(W_2(\bar \F_p))$ and let $\tilde s_1\in S_1(W_2(\bar \F_p))$ be its image. Let $s, s_1$ be the corresponding points modulo $p$, i.e., the corresponding geometric points. Then, $\pi_1$ restricts to a log smooth morphism
		$$
		\pi_{1,s}:(\sC,\sD)|_{\tilde{s}}\rightarrow (\sC_{1},\sD_{1})|_{\tilde s_1}
		$$
		over $W_2(\bar \F_p)$. By \cite[Corollary 1.2]{Sh}, there is an isomorphism of functors 
		$$
		\pi_{1,s}^*\circ C^{-1}_{(\sC_{1,s_1},\sD_{1,s_1})\subset (\sC_{1,\tilde s_1},\sD_{1,\tilde s_1})}\cong C^{-1}_{(\sC_{s},\sD_{s})\subset (\sC_{\tilde s},\sD_{\tilde s})}\circ \pi_{1,s}^*.
		$$
		This allows us to pull back a periodic parabolic Higgs-de Rham flow over $(\sC_{1,s_1},\sD_{1,s_1})$ initializing $(\sE_{1,s_1},\Theta_{1,s_1})$ to a periodic parabolic Higgs-de Rham flow over $(\sS_{s},\sD_{s})$ initializing $(\sE_{s},\Theta_{s})$ of \emph{the same period}. Therefore, $(E,\theta)$ is periodic with respect to the spreading-out $(\sC,\sD,\sE,\Theta)$. As the morphism $S\to S_2$ is smooth, there is a proper closed subscheme $Z_2\subset S_2$ such that $S\to S_2-Z_2$ is \emph{surjective}. Let $s_2\in S_2-Z_2$ be a geometric point and $\tilde s_2\to S_2$ be any $W_2$-lifting of $s_2$. Pick any inverse image $s\in S$ of $s_2$ as well as a $W_2$-lifting $\tilde s\to S$ over $s_2$ by smoothness. By construction, there is a natural isomorphism 
		$$
		(\sC,\sD,\sE,\Theta)|_{\tilde s}\cong (\sC_2,\sD_2,\sE_2,\Theta_2)|_{\tilde s_2}.
		$$
		So $(E,\theta)$ is also periodic with respect to the spreading-out $(\sC_2,\sD_2,\sE_2,\Theta_2)$ as claimed. 
	\end{proof}
 We conclude this section with some preliminary properties of periodic objects over $\C$.
	\begin{corollary}\label{vanishing chern classes and semistability}
		Assume $C$ is a smooth projective curve over $\C$ and $D\subset C$ is a reduced divisor. Let $(E,\theta)$ be a periodic parabolic Higgs bundle over $C_{\log}$. Then $(E,\theta)$ is parabolic semistable of degree zero. 
	\end{corollary}
	\begin{proof}
		As semistability is an open condition, the proposition follows from Proposition \ref{zero par degree in char p}.
	\end{proof}
A parabolic de Rham bundle $(V,\nabla,Fil)$ is said to be irreducible if $V$ does not contain nontrivial proper $\nabla$-invariant subbundle. Let $U$ be a smooth curve defined over $\bar \Q\subset \C$. To an irreducible periodic de Rham bundle over $U$, we are going to associate a family of representations of $\pi_1(U_{an})$.  Let us consider the problem first for an irreducible periodic parabolic de Rham bundle $V=(V,\nabla,Fil)$ over $C_{\log}=(C,D)$, which is assumed to be defined over a number field $K\subset \C$. Also, for simplicity, we shall assume the parabolic structure of $V$ is trivial (otherwise, we may achieve this by a finite cover $\pi: C'\to C$, \'etale over $U$).

Fix an isomorphism $\C\cong \C_p$ for each $p$.  Choose and then fix a spreading-out $(\sC,\sD, \sV,\nabla,\mathcal{F}il)$ of $(C,D,V,\nabla,Fil)$ over $S=\Spec\sO_K[\frac{1}{M}]$ satisfying the following properties:
\begin{itemize}
	\item [(i)] $\sC$ is smooth projective over $S$, and $\sD\subset \sC$ an $S$-relative simple normal crossing divisor;
   \item[(ii)] For each geometric point $s\in S$, $\Char\ k(s)>\rank(V)$ and there is a unique $W_2(k(s))$-lifting $\tilde s\to S$;
   \item[(iii)] The reduction $(\sV,\nabla,\mathcal{F}il)_s$ is periodic of period $f$, for any geometric point $s\in S$. 
\end{itemize}
Let $P_K$ be the set of finite places of $K$ and the chosen infinite place $K\subset \C$. 
\begin{proposition}
	Notations and assumptions as above.  Modulo finitely many finite places of $K$ and modulo a Tate twist, there is a unique family $\rho^V$ of representations of $\pi(U_{an})$ indexed by $P_K$ attached to $V$.
\end{proposition}
\begin{proof}
For the infinite place, $V$ defines an irreducible $\C$-local system over $U_{an}$, that is, a representation
$$
\rho^V(\infty): \pi_1(U'_{an}) \to \Gl(\C).
$$
Using Lemma \ref{openess of completely reducibility} (iii) below, $\Gr(V)$ is stable. It follows from \cite[Lemma 7.1]{LSZ}, $Fil$ coincides with the Simpson filtration $Fil_S$ up to a Tate twist (see \cite[Proposition 6.9]{LSZ} for the notion of \emph{Simpson filtration}). So we assume $Fil=Fil_S$ in the following.  Fix $s\in S$ which represents a finite place of $K$.  Then Proposition \ref{black hole principle} asserts that each Higgs term in a periodic de Rham-Higgs flow defining the property (iii) is stable. Again by \cite[Lemma 7.1]{LSZ}, one may actually take the grading functor with respect to the Simpson filtration in the periodic de Rham-Higgs flow. This makes a unique choice of periodic flow initializing $(\sV,\nabla,\mathcal{F}il)_s$ (without changing the period). Now by \cite[Theorem 1.2]{LSZ} and the Fontaine-Laffaille-Faltings correspondence \cite{Fa88}, we obtain a representation as follows:
$$
\rho^V(s): \pi_1(U_{an})\to \widehat{\pi_1(U_{an})}\cong \pi_1(U_{\C_p})\to \Gl(\F_{p^f}).
$$
By the $p$-adic Simpson correspondence \cite{Fa05}, $\rho^V(s)$ is irreducible.

\end{proof}

Let $V=(V,\nabla,Fil)$ be an irreducible periodic de Rham bundle over a smooth complex curve $U=C-D$.  Periodicity implies $\nabla$ is globally nilpotent (with respect to any spreading-out of $V$). The celebrated monodromy theorem of N. Katz \cite[13.0.1]{Kat70} asserts that $\nabla$ has only regular singularities and rational residues at $D$. Therefore, Lemma \ref{canonical parabolic extension} gives us a unique adjusted parabolic de Rham bundle $\bar V=(\bar V,\bar \nabla,\overline{Fil})$ over $C_{\log}$, which is moreover irreducible. 
\begin{conjecture}(Extension conjecture)\label{extension conjecture}
Let $U$ be a smooth complex curve. The canonical parabolic extension of an irreducible periodic de Rham bundle over $U$ is again periodic. 	
\end{conjecture} 
\begin{remark}\label{simpson filtration}
	Let $(V,\nabla)$ be an irreducible algebraic connection over $U$. Suppose there is some Hodge filtration $Fil$ on $V$ making the de Rham bundle $(V,\nabla,Fil)$ periodic. Then the conjecture predicts that, up to a shift of index, $Fil$ must be the restriction of the Simpson filtration\footnote{The existence of the Simpson filtration for an adjusted parabolic connection over $C_{\log}$ is shown by the same argument for \cite[Proposition 6.9]{LSZ}.} on the canonical parabolic extension of $(V,\nabla)$ over $C_{\log}$. Indeed, assuming the conjecture, the periodicity of $(V,\nabla,Fil)$ implies the periodicity of its canonical parabolic extension $(\overline V,\overline \nabla,\overline{Fil})$. By Lemma \ref{openess of completely reducibility} (iii), $Gr_{\overline{Fil}}(\overline{V},\overline{\nabla})$ is parabolic stable. Using the argument in the proof of Lemma 4.1 \cite{LSZ}, one concludes that $\overline{Fil}$ must be the Simpson filtration up to a shift of index. 
\end{remark}

Now we assume $U$ and $V$ are defined over $\bar \Q\subset \C$. Then the pair $(C,D)$ as well as the canonical extension $\bar V$ are also defined over $\bar \Q$.  Assume Conjecture \ref{extension conjecture}. Then we may take some finite cover $\pi: C'\to C$ defined over $\bar \Q$, \'etale over some nonempty open subset of $U$, such that $\pi^*_{\ppar}\bar V$ is of trivial parabolic structure and periodic (Lemma \ref{periodic parabolic pulls back to parabolic}). Therefore, the above discussion about an irreducible periodic logarithmic de Rham bundle applies.  Without assuming the conjecture, we may still have the \emph{preperiodicty} for $(\pi^*_{\ppar}\bar \sV)_s$ over any geometric point $s\in S$. The proof for preperiodicity is similar to that of \cite[Corollary 6.8]{LSZ}. Since we may always use the Simpson filtration defining the preperiodic de Rham-Higgs flow over $s$, we obtain an essentially unique representation of $\pi_1(U'_{an})$ into $\Gl(k(s))$ (where $U'$ is a finite \'etale cover of some nonempty open subset of $U$).

As we shall see, there is an apparent asymmetry between these two notions of periodic objects. Let $\DR(U)$ be the category of de Rham bundles over $U$. It is the category $\DR^{\ppar}(\sC_{\log}/S)$ in \S3, where $S=\Spec \C$ and $\sC_{\log}=(U,\emptyset)$. Similarly, we let $\mathrm{HG}(U)=\mathrm{HG}^{\ppar}(\sC_{\log}/S)$ be the category of graded Higgs bundles over $U$. The periodic de Rham (resp. Higgs) bundles form the subcategory $\PDR(U)$ (resp. $\PHG(U)$) of $\DR(U)$ (resp. $\mathrm{HG}(U)$). Note the grading functor restricts to $\Gr\colon \PDR(U)\to \PHG(U)$. This is because by taking the grading, a periodic de Rham-Higgs flow with the initial de Rham term $V$ gives rise to a periodic Higgs-de Rham flow with initial Higgs term $\Gr V$.
	
	\begin{proposition}\label{Tannakian cat}
		Let $U$ be a smooth complex curve. Then the category $\PDR(U)$ of periodic de Rham bundles over $U$ is a neutral Tannakian category over $\C$.
	\end{proposition}
	Before beginning the proof, we remind the reader that the category $\DR(U)$ is not abelian.  In particular, although $\MIC(U)$ is abelian, adding the filtration (and demanding compatibility with the filtration for morphisms) destroys the property of being an abelian category.
	\begin{proof}
		By \cite[Theorem 3.2]{LSZ} and \cite[Theorem 2.1]{Fa88}, it follows that the subcategory $\PDR(U)$ is abelian. To show the $\C$-linear abelian category $\PDR(U)$ is neutral Tannakian over $\C$, we use a criterion furnished by Deligne \cite[Proposition 1.20 ]{DM}. Fix a base point $x\in U$. Taking the fiber of an object over $x$ provides the fiber functor $\omega_{DR}\colon \PDR(U)\to \mathrm{Vect}_{\C}$. Note that $\PDR(U)$ is a rigid tensor category. By Lemma \ref{inverse Cartier commutes with tensor product}, $\PDR$ is a tensor subcategory and obviously a rigid tensor subcategory (namely, the dual of a periodic object is again periodic). A unit object in $\PDR(U)$ is given by $(\sO_U,d,Fil_{tr})$. It is clear that the natural morphism $\C$ to the endomorphism of the unit object is an isomorphism. It remains to show the above fiber functor is faithful and exact. Let $\MIC(U)$ be the category of flat bundles over $U$. Then, the forgetful functor of abelian categories $F: \PDR(U)\to \MIC(U)$ is exact and faithful. By Riemann-Hilbert correspondence, the fiber functor $\omega_x: \MIC(U)\to \mathrm{Vect}_{\C}$ is exact and faithful. As $\omega_{DR}=\omega_x\circ F$, $\omega_{DR}$ is also exact and faithful.  
	\end{proof}
	We shall come back to the discussion of categorical properties of $\PDR(U)$ after we establish our main discovery: de Rham periodicity theorem.

	\section{De Rham periodicity theorem}\label{section:proof} \label{section:geometric_objects}
	In this section, we shall establish Theorem \ref{main result}. Let $U$ be a smooth complex curve and $f: X\to U$ be a smooth projective morphism. Let $H_{dR}$ be a Gau{\ss}-Manin system associated to $f$. By Deligne's semisimplicity theorem, $H_{dR}$ decomposes into direct sum of irreducible factors. Say $H_{dR}=\oplus_i H_{i}$. Take a spreading-out of the decomposition, defined over some integral noetherian scheme $S$. Then, it follows from \cite[Theorem 6.2]{Fa88} and \cite[Proposition 3.3]{LSZ} that $H_{dR,s}$ is one-periodic for any geometric point $s$ (with respect to any $W_2$-lifting of $s$ inside $S$) away from a proper closed subscheme $Z\subset S$. Although the isomorphism class of the total sum $H_{dR,s}$ is invariant under the action of the flow operator $C^{-1}\circ\Gr$, one loses the insight about its action on the set of direct factors $\{H_{i,s}\}$. For instance, suppose there is an isomorphism $\alpha: A\oplus B\cong A\oplus B$; then it may very well happen that $\alpha(A)$ is isomorphic neither to $A$ nor to $B$.

	Our key idea is to extend the decomposition over the compactification $C$, so that the stability of irreducible factors of the associated graded parabolic Higgs bundles gives us the necessary control of the action in question. Recall that, by Lemma \ref{canonical parabolic extension}, one has the canonical parabolic extension $\bar H_{dR}$ of $H_{dR}$ over $C_{\log}$. 
	
	\begin{definition}
	Let $k$ be an algebraically closed field, and $C$ a smooth projective curve over $k$. A parabolic connection $(V,\nabla)$ over $C_{\log}/k$ is said to be \emph{irreducible} if there does not exist a nontrivial proper $\nabla$-invariant subbundle $W\subset V$ whose parabolic degree equals that of $V$. A parabolic de Rham bundle $(V,\nabla,Fil)$ over $C_{\log}/k$ is said to irreducible if $(V,\nabla)$ is irreducible.
	\end{definition}
	We maintain the assumptions of the beginning of the section. 
	\begin{lemma}\label{complete reducibility}
		Let $\bar H_{dR}$ be the canonical parabolic Gau\ss-Manin system attached to $f$. Then $\bar H_{dR}$ is completely reducible in the category of parabolic de Rham bundles over $C_{\log}$. Namely, there is a decomposition into a direct sum of parabolic de Rham bundles over $C_{\log}$: 
		$$
		\bar H_{dR}=\bigoplus_{i=1}^{n}\bar H_i^{\oplus m_i}, \quad m_i\geq 1,
		$$
		where $\bar H_i\subset \bar H_{dR}$ is an irreducible parabolic de Rham subbundle for each $i$, and $\{\bar H_i\}$s are pairwisely non-isomorphic. 
	\end{lemma}
	\begin{proof}
		The lemma is a consequence of \emph{Deligne's semisimplicity theorem} \cite{De71}, \cite{De84}\footnote{This formulation of semisimplicity theorem follows from the argument of \cite[Proposition 1.13]{De84}. More precisely, Deligne states a semisimplicity result for weight 0 polarized variations of Hodge structures, but the weight never enters into the argument. See also \cite[Theorem 2.1]{Mol} and the succeeding discussion on page 332.}: for each $i\geq 0$, there is a decomposition
		\begin{equation*}\label{eqn:decomposition}\displaystyle 
			R^i f_{*}\C= \bigoplus_{j=1}^{l}\mathbb{V}_j\otimes W_j,
		\end{equation*}
		where the $\mathbb{V}_j$s are mutually non-isomorphic irreducible complex local systems over $U_{an}$ and the $W_j$s are complex vector spaces. Moreover, each $\mathbb{V}_j$ carries a structure of $\C$-VHS and $W_j$ a structure of constant $\C$-VHS, unique up to a shift of bigrading, such that the above equation is an equality of $\C$-VHS.
		
		So by the Riemann-Hilbert correspondence, the Gau{\ss}-Manin connection $\nabla^{GM}$ decomposes into irreducible ones. Moreover, because of the $\C$-VHS property in the Deligne's decomposition, the Hodge filtration $F_{hod}$ splits accordingly-it suffices to write the constant $\C$-VHS structure on $W_i$ into a direct sum of Tate twists. Therefore, $H_{dR}=(H_{dR},\nabla^{GM},F_{hod})$ decomposes into irreducible ones in the category of de Rham bundles over $U$. Then an application of Lemma \ref{canonical parabolic extension} concludes the proof. 
	\end{proof}
	
	\begin{proposition}\label{motivic objects are determined by its generic fiber}
		Let $(V,\nabla,Fil)$ be a motivic de Rham bundle over $U$. Then its canonical parabolic extension is isomorphic to a direct summand of a parabolic Gau{\ss}-Manin system over $C_{\log}$. 
	\end{proposition}
	\begin{proof}
		By definition, there exists some nonempty Zariski open subset $U'\subset U$ and a smooth projective morphism $f: X'\to U'$ such that $V=(V,\nabla,Fil)$ is isomorphic to a de Rham subbundle of a Gau{\ss}-Manin system $H_{dR}$ associated to $f$. By Lemma \ref{canonical parabolic extension}, the canonical parabolic extension of $V|_{U'}$ is a parabolic de Rham subbundle of $\bar H_{dR}$. Because of the complete reducibility of $\bar H_{dR}$ (Lemma \ref{complete reducibility}), it follows that the canonical parabolic extension of $V|_{U'}$, which is just the canonical parabolic extension of $V$, is a direct factor of $\bar H_{dR}$. The proposition is proved. 	
	\end{proof} 
	
	We have the following parabolic Higgs version of Deligne's semisimplicity theorem. 
	\begin{lemma}\label{polystablility}
		Let $\bar H_{dR}=\bigoplus_i \bar H_i^{\oplus m_i}$ be the decomposition in Lemma \ref{complete reducibility}. The associated graded parabolic Higgs bundle $\Gr \bar H_{dR}$ is polystable of parabolic degree zero and it decomposes as
		$$
		\Gr \bar H_{dR}=\bigoplus_{i=1}^n (\Gr \bar H_i)^{\oplus m_i}
		$$
		with each $\Gr \bar H_i$ parabolic stable, and $\{\Gr \bar H_i\}$s are pairwisely non-isomorphic as graded parabolic Higgs bundles.
	\end{lemma} 
	\begin{proof}
		The proof is transcendental (though the statement is purely algebraic), we shall use the Simpson correspondence over quasi-projective curves \cite[Main Theorem, p. 755]{S90}. We need to make use of the Hodge metric on the $\C$-VHS $\V=R^if_{*}\C$ in the proof of Lemma \ref{complete reducibility}. By the Griffiths's curvature formula \cite[Lemma 7.18]{Sch}, the Hodge metric on $\V\otimes \sC^{\infty}(U)$ is harmonic. By Schmid's calculation (see \cite[Theorem 6.6]{Sch}), it is tame (compare  \cite[Proposition 2.1]{S90}). Then one uses \cite[Theorem 5]{S90} (and its proof, which is an application of the Chern-Weil formula), to conclude the parabolic polystability of $\Gr \bar H_{dR}$. It decomposes in the claimed form for the following reason: the orthogonal complement with respect to a Hermitian-Yang-Mills metric of a graded parabolic Higgs subbundle (=filtered regular subsystem of Hodge bundles in \cite{S90}) of degree zero is again \emph{graded} (i.e., $\C^{\times}$-fixed). 
	\end{proof}
	
	Now we study some consequences of the Deligne's semisimplicity theorem in arithmetic. We first record a lemma, which allows us to carry the decomposition over $\C$ to various mod $p$ reductions.

	\begin{lemma}\label{openess of completely reducibility}
		 Let $k$ be an algebraically closed field and $C$ a smooth projective curve over $k$. Then the following holds.
		\begin{enumerate}
			\item[(i)] Let $(V,\nabla)$ be a parabolic flat connection over $C_{\log}/\C$. If it is completely reducible with $m$ irreducible factors, then for any spreading-out $(\mathcal{C}_{\log},\mathcal{V},\nabla)$ over $S$, and for any geometric point $s$ of $S$ whose residue characteristic is sufficiently large, the parabolic flat connection $(\mathcal{V},\nabla)_s$ is completely reducible in the category parabolic flat connections on $(\mathcal C_{\log})_s$, and moreover has exactly $m$ irreducible factors.
			\item[(ii)] Let $(V,\nabla,Fil)$ be a periodic parabolic de Rham bundle over $C_{\log}/k$ which is completely reducible, where $k$ is either of positive characteristic or $\C$. Then $\Gr V$ is parabolic polystable. The number of parabolic stable factors of $\Gr V$ is equal to the number of irreducible factors of $(V, \nabla,Fil)$.
			\item[(iii)] Let $(V,\nabla,Fil)$ be as (ii). Then it is irreducible if and only if $\Gr V$ is parabolic stable.
		\end{enumerate}
	\end{lemma}
	\begin{proof}
		We first prove (i). As $(V,\nabla)$ is the direct sum of irreducible factors, it suffices to show that if $(V,\nabla)$ is an irreducible parabolic flat connection, then for the above spreading-out and all $s$ of large residue characteristic, $(\mathcal V,\nabla)_s$ is irreducible in the category of parabolic flat connections on $(\mathcal C_{\log})_s$. We prove this by contradiction: assume there are points $s$ of $S$ of arbitrarily large residue characteristic such that $(\mathcal V,\nabla)_s$ has a parabolic flat subbundle of parabolic degree 0. 
		
		By the BIS correspondence, we may assume the parabolic structure of $V$ is trivial. Let $\textrm{Quot}^{(r,0)}_{\mathcal V/\sX/S}$ be the Quot-scheme parameterizing quotient sheaves of $\sV_{s}$ which are of degree zero and rank $r$. Set $\textrm{Quot}=\coprod_{r}\textrm{Quot}^{(r,0)}_{\mathcal V/\sX/S}$. Then there is a \emph{closed} subscheme $\textrm{Quot}^{\nabla}$ of $\textrm{Quot}$ parametrizing $\nabla$-invariant quotient sheaves. It is a projective scheme of finite type over $S$. Because $\textrm{Quot}^{\nabla} (s)$ is non-empty as long as a geometric point $s\in S$ has large residue characteristic, it follows from Hilbert's Nullstellensatz that $\textrm{Quot}^{\nabla}(\C)$ is non-empty, contradicting the original assumption.
		
		By (i), it suffices to show the case $\Char k=p$ in (ii). We may also assume $V=(V,\nabla)$ is irreducible. As $V$ is periodic, $\Gr V$ is periodic, hence is parabolic semistable by Proposition \ref{vanishing chern classes and semistability}. Assume that $\Gr V$ is not parabolic stable. Then there exists $F\subsetneq \Gr V $, a nontrivial and proper parabolic Higgs subbundle of degree zero. Spread out the whole picture over $S$. By periodicity, it follows that after replacing $S$ by an open subscheme, there exists a natural number $f$ such that for all geometric points $s\in S$, $C^{-1}(\Gr\circ C^{-1} )^{f-1}(\sF_s)$ gives a nontrivial and proper parabolic de Rham subbundle of parabolic degree zero in $\mathcal V_s$, contradiction. which contradicts (i).  
		
		The only-if direction of (iii) follows from (ii). Now we assume that $\Gr V$ is parabolic stable. Consider first $\Char k=p$. By Proposition \ref{black hole principle}, $(\Gr\circ C^{-1} )^{f-1}\Gr V$ is again parabolic stable. Now assume the contrary, that there exists a nonzero $\nabla$-invariant subbundle $W\subsetneq V$ of parabolic degree zero.
        Since there is an isomorphism:
        $$
        C^{-1}(\Gr\circ C^{-1} )^{f-1}\Gr V\cong V,
        $$
        it follows that $C(W)\subset C(V)$ gives rise to a parabolic Higgs subbundle in $(\Gr\circ C^{-1} )^{f-1}\Gr V$ of parabolic degree zero, contradicting the parabolic stability. When $k=\C$, we take a spreading-out $(\mathcal C_{\log},\mathcal V,\nabla,Fil)$ over $S$. As parabolic stability is an open condition (\cite[Proposition 2.8]{MY92}), it follows that after replacing $S$ by an open subscheme, for geometric points $s\in S$, $\Gr(\mathcal V_s)=(\Gr \mathcal V)_s$ is parabolic stable. Hence the char $p$ statement together with (i) confirms this case. This proves the if-direction of (iii).
	\end{proof}
	The next result generalizes a result of Faltings \cite[Theorem 6.2]{Fa88} in the following sense: In loc. cit., for a \emph{semistable} family (see \cite[Ch.VI, b)]{Fa88}), Faltings shows that the higher direct image of the constant crystal (under certain natural conditions) is a Fontaine module, which is equivalent to a \emph{one-periodic} Higgs-de Rham flow by the correspondence established in \cite{LSZ}. His proof is restricted to the realm of logarithmic geometry. Hence it seems hard to extend this result for a \emph{quasi-semistable} family. Moreover, it is false that it is one-periodic in general. However, we are going to show that \emph{periodicity}, rather than one-periodicity, is still preserved. This provides more evidence that periodicity is the correct context to study motivic de Rham/Higgs bundles.
	\begin{proposition}\label{one-periodicity}
	Let $f: X\to U$ be a smooth projective morphism, and let $\bar H_{dR}$ be a parabolic Gau{\ss}-Manin system over $C_{\log}$ associated to $f$. Then $\bar H_{dR}$ is periodic. 
	\end{proposition}
	\begin{proof}
		Let us consider first the case that $f$ has unipotent local monodromies. In this case, the parabolic structure of a parabolic Gau{\ss}-Manin system associated to $f$ is trivial. Let $\bar f$ be a semistable compactification of $f$. It follows from \cite[Theorem 6.2]{Fa88} and \cite[Proposition 3.3]{LSZ}, that the reduction at $p$ of the logarithmic Gau{\ss}-Manin system attached to $\bar f$ is one-periodic, when $p$ is large enough. In general, we shall only obtain a quasi-semistable compactification $\bar f$, to which we shall apply the semistable reduction: Let $N$ be the least common multiple of the multiplicities of irreducible components of closed fibers of $\bar f$. Take another reduced divisor $D'\subset U$ such that $N$ is divisble by $\deg D+\deg D'$. So there is an ample line bundle $L$ such that $L^{\otimes N}=\sO_C(D+D')$. Then one forms the cyclic covering $\pi: C'\to C$ branching along $D+D'$ determined by $L$. We equip $C'$ with the logarithmic structure determined by the reduced divisor $(\pi^{-1}(D+D'))_{\text{red}}$. As $\pi$ is totally ramified over each point in $D+D'$ and unramified away from $D+D'$, $\pi$ is logarithmic smooth and even log finite \'etale. Now we consider the following commutative diagram:
		\[
		\xymatrix{ Y''\ar[dr]_{\bar g''}\ar[r]^{\tau}& Y'\ar[d]_{\bar g'}\ar[r]^{\delta} &Y\ar[d]^{\bar f} \\
			&  C'\ar[r]_{\pi} & C,}
		\]
		where $Y'$ is the normalization of $Y\times_CC'$ and $\tau$ is a resolution of singularities so that the composite map $\bar g''$ is semistable. Set $U'=C-D-D'\subset U$, $V'=f^{-1}(U')$ and $f'=f|_{V'}: V'\to U'$. Clearly, $\bar H_{dR}$ is also the canonical parabolic extension of the $i$-th Gau{\ss}-Manin system associcated to $f'$.  The local monodromies of the family $\bar g''$ are unipotent and hence the $i$-th logarithmic Gau{\ss}-Manin system $\bar H'_{dR}$ of $\bar g''$ is the canonical parabolic extension. By Lemma 3.7 \cite{IS}, the canonical parabolic extension is stable under logarithmic base change $(C', \pi^{-1}(D+D'))\to (C,D+D')$. It follows that
		$$
		\pi_{\ppar}^*\bar H_{dR}\cong \bar H'_{dR}
		$$
		as parabolic flat bundles. In fact, the above is an isomorphism of parabolic de Rham bundles since the Hodge filtration pulls back over the generic point.
		
		Now we spread out the above diagram of the semistable reduction and consider its reduction over a geometric point $s$ in the spreading-out of $C$ with $\mathrm{char}(k(s))$ sufficiently large (in particular, $\mathrm{char}(k(s))>N$).  By the one-periodicity of $\bar H'_{dR}$, there is an isomorphism of logarithmic flat bundles $\psi: C^{-1}\circ \Gr \bar H'_{dR,s}\to \bar H'_{dR,s}$.
	 Here $C^{-1}$ refers to any fixed $W_2(s)$-lifting $\tilde s\hookrightarrow S$. By Lemma \ref{parabolic inverse cartier commutes with pullback},
		$$
		C^{-1}\Gr \bar H'_{dR,s}=C^{-1}\pi_{\ppar}^*\Gr \bar H_{dR,s}=\pi_{\ppar}^* C^{-1}\Gr \bar H_{dR,s}.
		$$
	 Hence we are done if $\psi$ happens to be $G$-equivariant, since by the BIS correspondence it descends to an isomorphism of parabolic flat bundles from $C^{-1}\circ \Gr \bar H_{dR,s}$ to $\bar H_{dR,s}$. However, $\psi$ is not necessarily $G$-equivariant: Suppose that $\bar H_{dR,s}$ is irreducible. One has the following simple but important observation:
	 \begin{claim}\label{scalar}
	 $\Hom(C^{-1}\circ \Gr \bar H'_{dR,s},\bar H'_{dR,s})\cong k$.	
	 \end{claim}
	 \begin{proof}
	 This is a statement over an algebraically closed field $k=k(s)$ of positive characteristic. We note that the graded logarithmic Higgs bundle $\Gr \bar H'_{dR,s}$ is stable, instead of merely semistable. Hence, one has
	 $$
	 \End(\Gr \bar H'_{dR,s})\cong k.
	 $$
	 Since $C^{-1}$ is an equivalence of categories, it follows that
	 $$
	 \End(C^{-1}\circ \Gr \bar H'_{dR,s},C^{-1}\circ \Gr \bar H'_{dR,s})\cong k. 
	 $$
	 Fix an isomorphism $C^{-1}\circ \Gr \bar H'_{dR,s}\to \bar H'_{dR,s}$. Then one obtains the desired claim. 
	 \end{proof}
	 Therefore, $\sigma$ acts on $\psi$ by multiplication by $\zeta^l$ for some $0\leq l\leq N-1$. As $C^{-1}$ is semilinear, $\sigma$ acts on the induced isomorphism
		$$
		C^{-1}\circ \Gr\psi: (C^{-1}\circ \Gr)^2\bar H'_{dR,s} \to C^{-1}\circ \Gr\bar H'_{dR,s}
		$$
		by multiplication by $\zeta^{lp}$. It follows that $\sigma$ acts by multiplication by $\zeta^{l(1+p+\cdots +p^{f-1})}$ on the composite isomorphism 
		$$
		\psi\circ C^{-1}\circ \Gr\psi\circ\cdots \circ(C^{-1}\circ \Gr)^{f-1}\psi: (C^{-1}\circ \Gr)^{f}\bar H'_{dR,s} \to \bar H'_{dR,s}.
		$$		  
Let $p$ be large enough so that $(p,N)=1$. Let $1\leq q\leq N-1$ be the unique number which is congruent to $p$ modulo $N$. Set $d=\gcd(N,q-1)$, and write $N=dN'$, $q-1=dq'$. Let $k$ be the largest integer such that $d^k$ divides $q'$. 
\begin{claim}\label{bounded period}
$N$ divides the number $1+p+\cdots+p^{f-1}$ for $f=\phi(Nd^{k+1})$. 
\end{claim}
\begin{proof}
Indeed, since $(q,N)=(p,N)=1$ and $(q,d)=1$, it follows from Fermat's little theorem that $Nd^{k+1}=N'd^{k+2}$ divides $q^{f}-1$. As $q-1=q'd$ divides $q^f-1$ in any case, it follows that 
$$
\mathrm{lcm}(N'd^{k+2}, q'd)=\frac{N'd^{k+3}q'}{\gcd(N'd^{k+2}, q'd)}=N'd^2q'=N(q-1)
$$
divides $q^f-1$. Hence $N$ divides $\frac{q^f-1}{q-1}$. The claim follows. 
\end{proof}
So the composite isomorphism $\psi\circ C^{-1}\circ \Gr\psi\circ\cdots (C^{-1}\circ \Gr)^{f-1}\psi$ for $f=\phi(Nd^{k+1})$ is $G$-equivariant. Consequently, $(C^{-1}\circ \Gr)^{f}\bar H_{dR,s} \cong \bar H_{dR,s}$ as parabolic flat bundles, that is, $\bar H_{dR,s}$ is periodic with period a divisor of $\phi(Nd^{k+1})$, which is in turn a divisor of $\phi(N\cdot(N-2)!)$, a number independent of $p$. In order to handle the general case, we shall look into the structure of $\bar H'_{dR,s}$. First of all, Lemma \ref{complete reducibility} says that $\bar H'_{dR}$ is completely reducible. Then by Lemma \ref{openess of completely reducibility} (i), $\bar H'_{dR,s}$ remains completely reducible for $\mathrm{char}(k(s))$ sufficiently large. So it is a direct sum of irreducible logarithmic flat connections of degree 0, viz.
$$
\bar H'_{dR,s}\cong \bigoplus_i H_i^{\oplus m_i},
$$
with each $H_i$ irreducible and $\{H_i\}$s pairwisely non-isomorphic. Choose and then fix such an isomorphism. Then the $G$-action on $\bar H'_{dR,s}$ is transported over $\bigoplus_i H_i^{\oplus m_i}$. Clearly, each summand $H_i^{\oplus m_i}$ is fixed by the $G$-action. We may write $$H_i^{\oplus m_i}=H_i\otimes_k\Hom(H_i, H_i^{\oplus m_i}),$$ where $H_i$ in the tensor product is endowed with the trivial $G$-action. Then we decompose the multiplicity space $\Hom(H_i, H_i^{\oplus m_i})$ into a direct sum of eigen-spaces for the induced $G$-action. It follows that we may write $H_i^{\oplus m_i}=\oplus_{j=1}^{m_i}H_{ij}$ with $\sigma$ acting on $H_{ij}$ by multiplication by $\zeta^{l_{ij}}$. On the other hand, the existence of $\psi$ implies that there is also an isomorphism $C^{-1}\circ \Gr \bar H'_{dR,s}\cong \bigoplus_i H_i^{\oplus m_i}$. The previous argument also gives an isomorphism of $G$-modules:
$$
C^{-1}\circ \Gr\bar H'_{dR,s}\cong \bigoplus_{i,j} H_{ij},
$$
with $\sigma$ acting on $H_{ij}$ by multiplication by $\zeta^{l'_{ij}}$. Fix these $G$-isomorphisms. Then we consider the following (possibly new) isomorphism 
$$
\psi'=\bigoplus_{i,j}\psi_{ij}: C^{-1}\circ \Gr\bar H'_{dR,s}\to \bar H'_{dR,s},
$$  
induced by the identity map on $\bigoplus_{i,j} H_{ij}$. Then the natural $G$-action on $\psi'$ takes the following form
$$
\sigma(\bigoplus_{i,j}\psi_{ij})=\bigoplus_{i,j}\zeta^{l_{ij}-l'_{ij}}\psi_{ij}.
$$
The argument for an irreducible $H'_{dR,s}$ implies that the composite isomorphism $\psi'\circ C^{-1}\circ \Gr\psi'\circ\cdots \circ(C^{-1}\circ \Gr)^{f-1}\psi'$ is $G$-equivariant for $f=\phi(N\cdot(N-2)!)$. This completes the proof.
\end{proof}
\begin{remark}
It is possible to improve the integer $f$ in Claim \ref{bounded period}. The integer $\phi(N\cdot(N-2)!)$ may well serve as a period for \emph{any} $\bar H_{dR}$ whose weights are contained in $\frac{1}{N}\Z$.  It is an interesting question to find the optimal one in a specific situation.  However, we remind the reader that $f$ cannot be taken to be one in the following situation: There is a point in $D$ such that the set of nonzero weights of $\bar H_{dR,s}$ is simply $\{\frac{1}{N}\}$ with $\phi(N)\neq 1$. Indeed, since the grading functor does not change the set of weights,  it follows from Proposition \ref{change of par weights under inverse Cartier} that the set of nonzero weights of $C^{-1}\circ \Gr \bar H_{dR,s}$ at that point is $\{\left\langle \frac{p}{N}\right\rangle\}$. Therefore, the period for $\bar H_{dR,s}$ must be a multiple of $\phi(N)$.  
\end{remark}

	The following result improves and generalizes the Higgs periodicity theorem for semistable families \cite[Theorem 1.3]{LS}.
	\begin{theorem}\label{periodicity of stable factors of Kodaria-Spencer system}
		Notation as Proposition \ref{one-periodicity}. Let $(E,\theta)=\Gr \bar H_{dR}$ be the associated graded parabolic Higgs bundle with the decomposition into a direct sum of graded parabolic stable factors by Lemma \ref{polystablility}
		$$
		(E,\theta)=\bigoplus_{i=1}^n (E_i,\theta_i)^{\oplus m_i}.
		$$
		Then $(E,\theta)$ is periodic with respect to the Hodge filtration of $\bar H_{dR}$. Moreover, the Hodge filtration induces periodicity for each stable factor $(E_i,\theta_i)$.
	\end{theorem}
	\begin{proof}
		The first statement follows directly from Proposition \ref{one-periodicity}. It remains to show the second statement. Concerning the periodicity of each stable factor of $(E,\theta)$, one may reduce the general case to the semistable case using the same technique as Proposition \ref{one-periodicity}. So we may assume that $f$ admits a semistable compactification. Then \cite[Theorem 2.8]{LS} asserts that each factor $(E_i,\theta_i)$ is periodic: Let $\Gr$ be the grading functor with respect to the Hodge filtration. The proof therein actually shows that for some fixed number $r$, $(\Gr\circ C^{-1})^{r}$ maps each graded stable factor $(E_i,\theta_i)$ to some graded stable factor which is isomorphic to $(E_i,\theta_i)$ after ignoring the grading structure (see Caution in the proof of \cite[Theorem 2.8]{LS}). Now we will provide a simpler proof of this fact, which also shows that the grading structure is preserved under the flow operator.
		
		Spread out $(C_{\log},f\colon X\rightarrow U)$ to  $(\mathcal C_{\log},\mathfrak{f}\colon \mathcal X\rightarrow \mathcal U)$ over an integral affine scheme $S$ which is of finite type over $\Z$. By shrinking $S$, we may assume that both the Gau{\ss}-Manin bundle and its associated graded are locally free, and that the Hodge to de Rham spectral sequence degenerates at $E_1$. Let $(\mathcal E,\Theta)$ be the spreading out of the associated graded Higgs bundle of the Gau{\ss}-Manin system. Shrinking $S$ further if necessary, we may assume that the direct sum decomposition of $(E,\theta)$ spreads out over $S$:
		$$
		(\mathcal E,\Theta)=\bigoplus_{i=1}^n (\mathcal E_i,\Theta_i)^{\oplus m_i},
		$$
		and that for all geometric points $s$ of $S$, $(\mathcal E_i,\Theta_i)_s$ is parabolic stable, and moreover, $(\mathcal E_i,\Theta_i)_s\ncong (\mathcal E_j,\Theta_j)_s$ for $i\neq j$.
		
		Pick a graded parabolic stable factor $( E_1,\theta_1)$ in $(E,\theta)$, which by assumption spreads out to a parabolic Higgs subbundle $(\mathcal E_1,\Theta_1)\subset (\mathcal E,\Theta)$. For a closed point $s$ of $S$, consider the orbits under the flow operator $(\Gr\circ C^{-1})^f$ on $(\mathcal E_1,\Theta_1)_s$. The set $\{(\Gr\circ C^{-1})^{ft}(\mathcal E_1,\Theta_1)_s\}_{t\in \N}$ modulo \emph{graded} isomorphism is finite. Therefore, there exists a pair $(t_1,t_2)$ with $0\leq t_1<t_2\leq M$ such that 
		$$
		(\Gr\circ C^{-1})^{t_1}(\mathcal E_1,\Theta_1)_s\cong (\Gr\circ C^{-1})^{t_2}(\mathcal E_1,\Theta_1)_s
		$$
		as graded parabolic Higgs bundles. Let $(\mathcal F,\eta)$ be the unique stable parabolic Higgs subbundle of $(\mathcal E,\Theta)$ whose fiber at $s$ is isomorphic to $(\Gr\circ C^{-1})^{t_1}(\mathcal E_1,\Theta_1)_s$, which is $(t_2-t_1)$-periodic over $(\mathcal C_{\log})_s$. Now we write $(\mathcal E,\Theta)=(\mathcal F,\eta)\oplus (\mathcal E',\Theta')$. Because of one-periodicity, $C^{-1}(\mathcal E,\Theta)_s$ is isomorphic to $(\bar {\mathcal H}_{dR})_s$, the fiber of $\bar {\mathcal H}_{dR}$ over $(\mathcal C_{\log})_s$. Using Lemma \ref{complete reducibility} and Lemma \ref{openess of completely reducibility} (i), we may have that the Hodge filtration on $(\bar {\mathcal H}_{dR})_s$ splits, in the sense that it is the direct sum of filtrations on irreducible parabolic flat connections. In particular, there is a decomposition of parabolic de Rham bundles:
		$$
		(C^{-1}(\mathcal E,\Theta)_s, F_{hod})=(C^{-1}(\mathcal F,\eta)_s,F_{hod})\oplus (C^{-1}(\mathcal E',\Theta')_s,F_{hod}).
		$$ 
		
		From this, we conclude that $(\mathcal E',\Theta')_s$ with respect to the Hodge filtration is $(t_2-t_1)$-periodic. Note that $(\mathcal E',\Theta')_s$ is polystable and has one less graded stable factor than $(\mathcal E,\Theta)_s$. A repetition of the above argument shows that each stable factor in $(\mathcal E,\Theta)_s$ is periodic with respect to the Hodge filtration. Moreover, their periods are all bounded by $M!$ for $M=\sum_{i=1}^nm_i$. As $s$ is arbitrary, this shows the periodicity, as desired. 
		%
		
	\end{proof}
	\begin{corollary}\label{geometric is periodic}
		Let $\bar H_{dR}=\bigoplus_{i=1}^{n}\bar H_i^{\oplus m_i}$ be the decomposition in Lemma \ref{complete reducibility}. Then each irreducible factor $\bar H_{i}$ is periodic.
	\end{corollary}
	\begin{proof}
		Let $(\sC_{\log},\bar {\mathcal H}_{dR}=\bigoplus_i \bar {\mathcal H}_{i}^{\oplus m_i})$ be a spreading-out over $S$. Then Lemma \ref{polystablility} and Lemma \ref{openess of completely reducibility} (i), (ii) show that the grading functor $\Gr$ is a bijective map from the set of isomorphism classes of parabolic de Rham subbundles to the set of isomorphism classes of graded parabolic Higgs subbundles, first in characteristic zero and then over $k(s)$ for any geometric points $s\in S-Z$. So the periodicity for $\bar H_i$ follows from that of $\Gr \bar H_i$, as established in Theorem \ref{periodicity of stable factors of Kodaria-Spencer system}.  
	\end{proof}
	Now we are able to prove the main result of this article. 
	\begin{proof}[Proof of Theorem \ref{main result}]
		Let $V=(V,\nabla,Fil)$ be a motivic de Rham bundle over $U$. By Proposition \ref{motivic objects are determined by its generic fiber}, the canonical parabolic extension of $V$ to $C$ is a direct factor of some parabolic Gau{\ss}-Manin system $\bar H_{dR}$. By Proposition \ref{Tannakian cat}, we may assume $V$ is an irreducible factor of $\bar H_{dR}$. By Corollary \ref{geometric is periodic}, the canonical parabolic extension of $V$ is periodic. Clearly, periodicity is preserved under the restriction from $C$ to $U$. Hence the periodicity for $V$ follows.  
	\end{proof}

	\section{Periodic de Rham conjecture}\label{section:conjecture}
	Let us fix a smooth complex projective curve $C$.  Corollary \ref{geometric is periodic} asserts that motivic parabolic de Rham bundles over $C_{\log}$ are periodic. The next lemma shows a nontrivial property for irreducible periodic objects; the corresponding property for motivic objects was previously known via the Simpson correspondence using the Hodge metric.
	\begin{lemma}\label{fullness}
		Suppose $V_i,i=1,2$ are irreducible periodic parabolic de Rham bundles over $C_{\log}/\C$. Then for any $\beta: \Gr V_1\to \Gr V_2$, a morphism of periodic parabolic Higgs bundles, there exists a unique $\alpha: V_1\to V_2$ of periodic parabolic de Rham bundles such that $\Gr \alpha=\beta$.
	\end{lemma}
	\begin{proof}
		By Lemma \ref{openess of completely reducibility} (iii), $\Gr V_i, i=1,2$ are parabolic stable. A nonzero $\beta$ implies that $\Gr V_i, i=1,2$ are isomorphic. Let $(\sC_{\log},\sV_1,\sV_2)$ be a spreading-out of $(C_{\log}, V_1, V_2)$ over $S$. Then the periodicity of $V_i, i=1,2$ implies that $(\sV_i)_s, i=1,2$ are isomorphic for any geometric point $s\in S-Z$. Thus $V_i, i=1,2$ are isomorphic as parabolic flat bundles. Write $V_i=(V_i,\nabla_i,Fil_i)$ for $i=1,2$. Actually, we have shown that the $\C$-linear vector space $\Hom((V_1,\nabla_1),(V_2,\nabla_2))$ is one-dimensional. It contains the $\C$-linear subspace $\Hom((V_1,\nabla_1,Fil_1),(V_2,\nabla_2,Fil_2))$ whose elements preserve the Hodge filtrations.
		\begin{claim}
		$\Hom((V_1,\nabla_1,Fil_1),(V_2,\nabla_2,Fil_2))=\Hom((V_1,\nabla_1),(V_2,\nabla_2))$.
		\end{claim}
		\begin{proof}
		Fix an isomorphism $\psi: V_1\to V_2$ belonging to $\Hom((V_1,\nabla_1),(V_2,\nabla_2))$. Then $\psi$ induces an isomorphism of graded parabolic Higgs bundles: 
		$$
		\Gr_{\psi^*Fil_2}(V_1,\nabla_1)\cong \Gr_{Fil_2}(V_2,\nabla_2).
		$$ 
		In particular, $\Gr_{\psi^*Fil_2}(V_1,\nabla_1)$ is parabolic stable. Now by the BIS correspondence (Proposition \ref{Biswas correspondence for Higgs/flat bundles in positive char}) and the proof of \cite[Lemma 4.1]{LSZ}, it follows that $\psi^*Fil_2=Fil_1$ since $\Gr_{Fil_1}(V_1,\nabla_1)$ is also parabolic stable by assumption. The claim follows.
		\end{proof}
		Again by the parabolic stability, the $\C$-linear space $\Hom(\Gr_{Fil_1}(V_1,\nabla_1),\Gr_{Fil_2} (V_2,\nabla_2))$ of graded parabolic Higgs bundles is one-dimensional. The natural $\C$-linear map 
		$$
		\Gr: \Hom((V_1,\nabla_1,Fil_1),(V_2,\nabla_2,Fil_2))\to \Hom(\Gr_{Fil_1}(V_1,\nabla_1),\Gr_{Fil_2} (V_2,\nabla_2)) 
		$$
		is non-zero. Therefore, The above map $\Gr$ must be an isomorphism. The lemma is proved.
	\end{proof}
	By abuse of notation, we use $C_{\log}$ to mean a log curve $(C,D)$ for \emph{some} (viz. not fixed) reduced effective divisor $D$. We define the category $\MDR(C_{\log})$ (resp. $\MHG(C_{\log})$) of motivic parabolic de Rham (resp. Higgs) bundles over $C_{\log}$, as the full subcategory of the category $\PDR^{ss}(C_{\log})$ (resp. $\PHG^{ss}(C_{\log})$) of completely reducible periodic parabolic de Rham bundles (resp. polystable periodic parabolic Higgs bundles) over $C_{\log}$. 
	\begin{proposition}
	The grading functor $\Gr: \PDR^{ss}(C_{\log})\to \PHG^{ss}(C_{\log})$ is fully faithful and restricts to an equivalence of categories $\Gr: \MDR(C_{\log})\cong \MHG(C_{\log})$.	
	\end{proposition}
	\begin{proof}
	The fully faithfulness property for $\Gr$ follows from Lemma \ref{fullness}.  The remaining is clear.
	\end{proof}
	We make the following conjecture.
	\begin{conjecture}(weak periodic de Rham conjecture)
	The natural inclusion functor 
	$$\MDR(C_{\log})\to \PDR^{ss}(C_{\log})$$ 
	is an equivalence. In other words, each irreducible periodic parabolic de Rham bundle over $C_{\log}$ is motivic.
	\end{conjecture}
	
Now we fix a nonempty subset $U\subset C$.  Recall the category $\PDR(U)$ (resp. $\PHG(U)$) of periodic de Rham (resp. Higgs) bundles over $U$ defined in \S4. After Theorem \ref{main result} (resp. Theorem \ref{periodicity of stable factors of Kodaria-Spencer system}), we define the category $\MDR(U)$ (resp. $\MHG(U)$) of motivic de Rham (resp. Higgs) bundles over $U$ as full subcategory of $\PDR(U)$ (resp. $\PHG(U)$).  
\begin{proposition}\label{equivalence}
	The functor $\Gr: \PDR(U)\to \PHG(U)$ is a functor of rigid tensor categories which is faithful and preserves $\bigoplus$, $\bigotimes$ and unit objects. It restricts to 
	$$\Gr: \MDR(U)\to \MHG(U)$$ of neutral Tannakian categories over $\C$.
	\end{proposition}
	
	\begin{proof}
		The first statement is clear except the faithfulness of $\Gr$. Let $\alpha: V_1 \to V_2$ be a morphism in $\PDR(U)$ such that $\beta=\Gr \alpha: \Gr V_1\to \Gr V_2$ equals zero. Take a spreading-out of $(U, V_1,V_2,\alpha)$ over $S$. So for each closed point $s\in S-Z$, the reduction $\beta_s$ at $s$ equals zero. The functors $C^{-1}$ and $\Gr$ map the zero morphism to the zero morphism. Hence by periodicity $\alpha_s=0$, which implies $\alpha=0$. So $\Gr$ is faithful. Clearly, the categories of geometric objects form rigid tensor subcategories. By Lemmata \ref{complete reducibility} and \ref{polystablility}, the categories are abelian. It follows that they are Tannakian subcategories.  
	\end{proof}
	
Combining the extension conjecture with the weak periodicity de Rham conjecture, we obtain the following conjecture. 	
\begin{conjecture}\label{periodic dR conjecture}
		Let $U$ be a smooth curve over $\C$. Let $\PDR^{ss}(U)$ be the full subcategory of $\PDR(U)$ consisting of completely reducible objects. Then
		$$
		\MDR(U)=\PDR^{ss}(U).
		$$
\end{conjecture}
In turn, this conjecture implies both the extension conjecture and the weak periodicity de Rham conjecture.  To support our conjecture, we make a classification of rank one objects in the category $\PDR(U)$. For a de Rham object $V=(V,\nabla,Fil)$, we define the Tate twist $V(n)$ to be $(V,\nabla,Fil[n])$.  
	\begin{proposition}\label{rank one case}
		A rank one object in $\PDR(U)$ is of form $L(n)$ where $L=(L,\nabla,F_{tr})$ with $F_{tr}$ is the trivial filtration and $(L,\nabla)$ torsion, namely, there is some positive integer $m$ such that $(L,\nabla)^{\otimes m}\cong (\sO_U, d)$. 
	\end{proposition}
	\begin{proof}
		Up to Tate twist, a rank one object in $\PDR(U)$ is of form $(V,\nabla, F_{tr})$. Periodic de Rham bundles in positive characteristic have nilpotent $p$-curvature. But a rank one connection having nilpotent $p$-curvature means that the $p$-curvature is zero. Thus $\nabla$ is a connection on $U$ with vanishing $p$-curvature for almost all $p$s. Thus, by the solution of Grothendieck-Katz $p$-curvature conjecture for rank one objects due to Chudnovsky-Chudnovsky and Andr\'{e} (see e.g. \cite[\S2.4 ]{Bost}), $(L,\nabla)$ is finite \'{e}tale trivializable. As it is of rank one, it can be trivialized by a finite cyclic covering of $U$. It follows that there is some $m\in \N$ such that $(L,\nabla)^{\otimes m}=(\sO_U,d)$.
	\end{proof}
	By the previous result, any periodic rank one de Rham bundle over $U$ is motivic: using the notion in the proof of Proposition \ref{rank one case}, $(L,\nabla)$ is a direct factor of $\pi_*(\sO_X,d)$, where $\pi: X\to U$ is a finite \'{e}tale cover trivializing $(L,\nabla)$ by pullback.

	Now we turn our attention to some special higher rank connections, namely those weak physically semi-rigid (WPSR) connections. By Riemann-Roch there are no interesting rigid connections over smooth curves in the classical sense. Our first task is to give a slight reformulation of the notion of WPSR connections.

	Let $M^{1,\Gamma}_{dR}(C_{\log})$ be the moduli space of rank one parabolic connections over $C_{\log}$ of degree zero and of fixed quasi-parabolic structure $\Gamma$. We also let $M^{1}_{dR}(C)$ be the moduli space of rank one connections on $C$. It has a natural action on $M^{1,\Gamma}_{dR}(C_{\log})$ by tensor product, which is simply transitive. Let $J$ be a set of Jordan normal forms whose eigenvalues are roots of unity (we call such a $J$ \emph{rational}). Consider the following (untwisted) de Rham moduli space
	$$
	M^{r,J}_{dR}(U):=\{(V,\nabla)|\rank(V)=r,\ \Res(\nabla)\cong J\}/\cong.
	$$
	Since $J$ is rational, $\det J=(\det J_1,\cdots, \det J_n)$ is also rational. For a connection $(V,\nabla)$ in the moduli, the adjusted rank one parabolic connection associated to $\det (V,\nabla)$ has a fixed quasi-parabolic structure on the punctures. Denote it by $\Gamma_{\det J}$. There is a natural map
	$$
	\overline\det: M^{r,J}_{dR}(U)\to M^{1,\Gamma_{\det J}}_{dR}(C_{\log}), \quad [(V,\nabla)]\mapsto [(\overline {\det V},\overline {\det \nabla})]. 
	$$ 
	Recall the following twisted de Rham moduli space
	$$
	M^{r,J}_{dR}(U)_{twisted}:=\{(V,\nabla)| \rank(V)=r,\ \Res(\nabla)\cong J\}/\thicksim,
	$$
	where $(W,\nabla)\thicksim (V,\nabla)$ if there exists some $(L,\nabla)$ over $C$ such that $(W,\nabla)\cong (V,\nabla)\otimes (L,\nabla)|_U$.
	\begin{lemma}\label{reformulation of WPSP connections}
		Fix any point $x_0$ in $M^{1,\Gamma_J}_{dR}(C_{\log})$. The fiber $\overline \det^{-1}(x_0)$ is a finite set if and only if the twisted de Rham moduli $M^{r,J}_{dR}(U)_{twisted}$ consists of finitely many points.
	\end{lemma} 
	\begin{proof}
		Consider the obvious map 
		$$
		\Psi: \overline \det^{-1}(x_0)\to M^{r,J}_{dR}(U)_{twisted}.
		$$
		Let $[V_i], i=1,2$ be two elements in $\overline \det^{-1}(x_0)$. Suppose $\Psi([V_1])=\Psi([V_2])$. Then there is an element $L$ in $M^{1}_{dR}(C)$ such that $\overline{V}_1\cong \overline{V}_2\otimes L$. It follows that
		$$
		\det \overline{V}_1\cong \det \overline{V}_2\otimes L^{\otimes r}
		$$
		As $[\det \overline{V}_i]=[\overline{\det (V_i)}]=x_0$, $L$ is $r$-torsion. Thus the fiber of $\Psi$ is finite. On the other hand, for any element $[V]$ in $ M^{r,J}_{dR}(U)$, we may always find some rank one connection $L$ over $C$ such that 
		$$
		[\det (\overline V)\otimes L^{\otimes r}]=x_0
		$$
		by the divisibility of $M^{1}_{dR}(C)$. So the map $\Psi$ is surjective. The lemma follows. 
	\end{proof}
	
	\begin{proposition}\label{CVHS on rigids}
		Let $J$ be a set of Jordan normal forms which is rational. Let $(V,\nabla)$ be an irreducible WPSR connection of type $J$ over $U$. Then the monodromy representation of $(V,\nabla)$ underlies a structure of $\C$-VHS. Up to a shift of index, the Hodge filtration on $V$ is the restriction to $U$ of the Simpson filtration on the canonical extension of $(V,\nabla)$ over $C_{\log}$. 
	\end{proposition}
	\begin{proof}
		Let $\V$ be the representation of $\pi_1(U_{an})$ attached to $(V,\nabla)$. By assumption, it is irreducible and weak physically semi-rigid in the sense of Katz. Let $(\overline E,\overline \theta)$ be the corresponding parabolic stable Higgs bundle over $C_{\log}$ under the Simpson correspondence \cite{S90}. Choose a sequence $t_i\rightarrow 1$, and assume that $t_i$s are not roots of unity. Then the representations $\V_i$ corresponding to $(\overline E,t_i\overline \theta)$ converges to $\V$. As $\V$ is WPSR, $\V_i$ must be isomorphic to $\V$ when $i$ is large enough. Thus, $(\overline E,t\overline \theta)\cong(\overline E,\overline \theta)$ for some $t\in \C^*$ of infinite order. It follows from Theorem 8 \cite{S90} that $\V$ underlies a $\C$-VHS. Now as $J$ is rational, there is a Hodge filtration $\overline{Fil}$ on the canonical parabolic extension $(\overline V,\overline \nabla)$ such that $Gr_{\overline{Fil}}(\overline V,\overline \nabla)=(\overline E,\overline \theta)$. The rest follows from Remark \ref{simpson filtration}.
	\end{proof}
	
	\begin{definition}
		Let $k$ be an algebraically closed field. Let $(E,\theta)$ be a parabolic Higgs bundle over $C_{\log}$ with parabolic structure $\Gamma$. It is said to be non-rigid if there exists a relative parabolic Higgs bundle $\Theta: \sE\to \sE\otimes \omega_{\sC_{\log}/S}$ of quasi-parabolic structure $\Gamma$, where $S$ is a connected algebraic variety over $k$ and $\sC=C\times_kS$, equipped with the log structure determined by the divisor $D\times_kS$, such that its restriction over one closed point $s_0\in S$ is isomorphic to $(E,\theta)$ and over another closed point $s_1\in S$ is non-isomorphic to $(E,\theta)$. Similarly, one defines non-rigid parabolic flat bundles over $C_{\log}$ with fixed quasi-parabolic structure.
	\end{definition}
	We need an important property about rigid/non-rigid adjusted parabolic connections in positive characteristic. 
	\begin{proposition}\label{rigidity lemma}
		Let $k$ be an algebraically closed field of char $p>0$. Let $C$ be a smooth projective curve over $k$. Let $(V,\nabla)$ be a parabolic connection in the category $\mathrm{MIC}^{\ppar}_{p-1,N}(C_{\log}/k)$ and $(E,\theta)$ the corresponding parabolic Higgs field by Theorem \ref{ov correspondence for parabolic objects}. Then 
		The following statements hold:
		\begin{itemize}
			\item[(i)] $(E,\theta)$ is $\theta$-semistable if and only if $(V,\nabla)$ is $\nabla$-semistable;
			\item[(ii)] Suppose $(E,\theta)$ is $\theta$-semistable. Then $(E,\theta)$ is rigid if and only if $(V,\nabla)$ is rigid.
		\end{itemize} 
	\end{proposition} 
	\begin{proof}
		
		By Theorem \ref{ov correspondence for parabolic objects}, it follows that the set of  $\theta$-invariant parabolic subsheaves of degree $d$ is in one-to-one correspondence to the set of $\nabla$-invariant parabolic subsheaves of $V$ of degree $pd$. From this, (i) follows.
		
		To show (ii), we first apply the BIS correspondence to reduce it to the case of trivial parabolic structure (and with a $G$-action). It means that, given a nontrivial deformation of a parabolic Higgs bundle $(E,\theta)$, we will obtain a nontrivial $G$-equivariant deformation of the logarithmic Higgs bundle, the parabolic pullback of $(E,\theta)$. It may well happen that in the deformation of $(E,\theta)$, the Higgs field could become non-nilpotent. Therefore the Ogus-Vologodsky correspondence is not sufficient to form a corresponding deformation for $(V,\nabla)$. However, one may use the theory of Chen-Zhu \cite{CZ15} to tackle this difficulty. We first note that the main result Theorem 1.2 loc. cit. admits a direct generalization to our logarithmic setting. Namely, by taking the invertible sheaf $\sL$ in \cite{CZ15} to be   $\omega_{C_{\log}/k}$, the article \cite{CZ15} yields the following statement: over an \'etale open neighborhood $U$ of the origin in the Hitchin base $B_{\omega_{C_{\log}/k}}$, there is a one-to-one correspondence between ($G$-equivariant) logarithmic Higgs bundles and ($G$-equivariant) logarithmic flat bundles which restricts to the one given by Ogus-Vologodsky's correspondence over the origin of $B_{\omega_{C_{\log}/k}}$. Now suppose $(E,\theta)$ is $\theta$-semistable and rigid. Assume the contrary that the corresponding $(V,\nabla)$ is non-rigid. We prove that $(E,\theta)$ must be also non-rigid. By (i), $V$ is $\nabla$-semistable. By \cite[Theorem 1.1]{Lan14}, there is a coarse moduli variety $M$ of semistable logarithmic flat connections containing the moduli point $[(V,\nabla)]$. As semistablity is an open condition, we may assume there is a \emph{non-constant} morphism $\phi: S\to M$ with $\phi(s_0)=[(V,\nabla)]$. Likewise, there is a coarse moduli variety $N$ containing the point $[(E,\theta)]$. By the result of Chen-Zhu \cite{CZ15}, there is an isomorphism $\mathcal{C}: M_U=M\times_B U\cong N\times_U B=N_U$.  Pulling back $\phi$ via the natural map $M_U\to M$ which is \'etale, we shall obtain a non-constant map $T=S\times_{M}M_U\to M_U$. Composing it with $M_U\stackrel{\mathcal{C}}{\to}N_U\to N$, we get a \emph{non-constant} map $T\to N$ so that $[(E,\theta)]$ lies in its image. Therefore, by the moduli interpretation, $(E,\theta)$ admits a nontrivial deformation, which is a contradiction. One proves the other direction in a completely symmetric way.
	\end{proof}

	\begin{proposition}\label{irreducible WPSR is periodic}
		Let $(V,\nabla)$ be a connection over $U$ as in Proposition \ref{CVHS on rigids}; in particular, $(V,\nabla)$ is WPSR. Assume that the determinant of $(V,\nabla)$ is torsion. Let $F_S$ be the Simpson filtration on the canonical parabolic extension $(\overline V,\overline \nabla)$ . Then the de Rham bundle $(V,\nabla, (F_S)|_{V})$ is periodic.
	\end{proposition}
	\begin{proof}
		We are going to show that the parabolic de Rham bundle $(\overline V,\overline \nabla,F_S)$ is periodic. By Proposition \ref{CVHS on rigids}, $(\overline E,\overline \theta)=Gr_{F_S}(\overline V,\overline \nabla)$ is the \emph{corresponding} parabolic Higgs bundle under the Simpson correspondence \cite[p. 721-722]{S90}. Therefore, $(\overline E,\overline \theta)$ is rigid. 
		
		Fix a spreading-out $(\mathscr{C},\scrD, \overline \scrV,\overline \nabla, \mathscr{Fil}_S)$ of $(C,D,\overline V,\overline \nabla,F_S)$ over $S$. A spreading-out of $(E,\theta)$ over $S$ is then given by $(\scrE,\Theta)=Gr_{\mathscr{F}_S}(\overline \scrV,\overline \nabla)$. Let $\Gamma_J$ be the quasi-parabolic structure of $(\overline V,\overline \nabla)$. By assumption, the canonical extension of $(\det V,\det \nabla)$ is a torsion, say of order $m$. Set $L_{dR,i}=(\det V,\det \nabla)^{\otimes i}$ for $0\leq i\leq m-1$. Let $\sM_{dR}^{r,\Gamma_J,\scrL_{dR,i}}(\mathscr{C}_{\log})$ be the moduli scheme of $S$-relative adjusted parabolic connections over $\mathscr{C}_{\log}$ of rank $r$, quasi-parabolic structure $\Gamma_J$ and of determinant $(L,\nabla)$, and $\sM_{Dol}^{r,\Gamma_J,\scrL_{Dol,i}}(\mathscr{C}_{\log})$ be the moduli scheme of parabolic semistable $S$-relative Higgs fields over $\mathscr{C}_{\log}$ of rank $r$, quasi-parabolic structure $\Gamma_J$ and of determinant $L_{Dol,i}=((\det V)^{\otimes i},0)$ \footnote{Such moduli schemes for parabolic connections/Higgs fields exist by using the techniques in \cite{Lan14}.}. After shrinking $S$ if necessary, we may assume, all geometric fibers of the scheme $\sM_{dR}^{r,\Gamma_J,\scrL_{dR,i}}(\mathscr{C}_{\log})$ over $S$ have the same number of irreducible components. The same holds for the Higgs moduli. Moreover, by assumption and Lemma \ref{reformulation of WPSP connections}, $\sM_{dR}^{r,\Gamma_J,\scrL_{dR,i}}(\mathscr{C}_{\log})(\C)$ consists of finitely many points for each $i$, and its cardinality equals to that of $\sM_{Dol}^{r,\Gamma_J,\scrL_{Dol,i}}(\mathscr{C}_{\log})(\C)$ by the Simpson correspondence. 
		
		Let $N$ be the least common multiple of the orders of eigenvalues of residues of $\nabla$, so that the parabolic weights of $\Gamma$ lie in $\frac{1}{N}\Z$. Pick a geometric point $s\in S$ such that the characteristic $\mathrm{char}(k(s))=p$ satisfies
		$$
		p-1\geq \max\{N,m\}.
		$$
		The reduction $(\overline \scrV,\overline \nabla)_s$ is an object in the category $\mathrm{PMIC}_{p-1,N}(\mathscr{C}_{\log,s}/k(s))$, where  $\mathscr{C}_{\log,s}$ is the reduction of  $\mathscr{C}_{\log}$ at $s$. Choose any $W_2(k(s))$-lifting $\tilde s$ inside $S$ and denote $C^{-1}_{s\subset \tilde s}$ for the corresponding parabolic inverse Cartier transform. By Proposition \ref{rigidity lemma}, $C^{-1}_{s\subset \tilde s}(\scrE_s,\Theta_s)$ is again rigid. However, its determinant is isomorphic to $\scrL_{dR,s}^{\otimes p}$. Because 
		$$
		\scrL_{dR,s}^{\otimes p}=(\scrL_{dR}^{\otimes p})_s=(\scrL_{dR,i})_s,
		$$ 
		where $0\leq i\leq m-1$ satisfies $i\equiv p\mod m$, it follows that $C^{-1}_{s\subset \tilde s}(\scrE_s,\Theta_s)$ comes from the reduction at $s$ of some parabolic connection belonging to $\sM_{dR}^{r,\Gamma_J,\scrL_{dR,i}}(\C)$. Therefore, $\Gr_{F_{S,s}}(C^{-1}_{s\subset \tilde s}(\scrE_s,\Theta_s))$ is again rigid. Repetition of the arguments shows that 
		$$
		(C^{-1}_{s\subset \tilde s}\circ \Gr_{F_{S,s}})^{\phi(m)}(\overline \scrV,\overline \nabla)_s
		$$ 
		is rigid and of determinant $(\scrL_{dR,1})_s$. Therefore, the operator $(C^{-1}_{s\subset \tilde s}\circ \Gr_{F_S})^{\phi(m)}$ sets up a self-map on the set $\sM_{dR}^{r,\Gamma_J,\scrL_{dR,1}}(\C)$ by the above process. The map is surjective, because we can reverse the process: note that $C_{s\subset \tilde s}(\overline \scrV,\overline \nabla)_s$ is rigid by Proposition \ref{rigidity lemma} and therefore must be reduction at $s$ of some rigid parabolic Higgs bundle over $C_{\log}$ belonging to  $\sM_{Dol}^{r,\Gamma_J,\scrL_{Dol,i}}(\mathscr{C}_{\log})(\C)$ for some $i$. By the Simpson correspondence, it is the grading of some rigid parabolic connection over $C_{\log}$. As the set is finite, the self-map must be bijective. So $(\overline V,\overline \nabla,F_S)$ is periodic as desired. The proposition is proved.
	\end{proof}
	
	Katz \cite{Kat96} has given a classification of (irreducible) WPSR local systems over $U$. Based on this, we are going to show the following
	\begin{proposition}\label{motivicity of rigids}
		Let $(V,\nabla)$ be as in Proposition \ref{irreducible WPSR is periodic}, i.e., $(V,\nabla)$ is an irreducible weakly physically semi-rigid flat connection on $U$. Then it is motivic in the sense of Definition \ref{motivic objects}.
	\end{proposition}
	\begin{proof}
		By \cite[Corollary 1.2.4]{Kat96}, we have $g(C)=0,1$. We divide the proof into two cases. Set $V=(V,\nabla)$.
		
		{\itshape $g(C)=0$.} The argument of \cite[9.4.2]{Kat96} (see p. 212-213, Ch. 9 \cite{Kat96}) implies that there exists an open subset $W\subset U$, a smooth affine hypersurface $f\colon Y\rightarrow W$ and a smooth partial compactification $Y\hookrightarrow Y'$ with a smooth proper extension $f'\colon Y'\rightarrow W$ such that $(V,\nabla)$ is a summand of:
		$$\text{Image}(H^i_{dR}(Y'|W)\rightarrow H^i_{dR}(Y|W)).$$
		(The smooth affine hypersurface in fact exists over $U$, but to guarantee a smooth normal crossings compactification one might need to shrink $U$ to a smaller Zariski open $W$.) By Deligne's semisimplicity theorem, it follows that $(V,\nabla)|_W$ is in fact a direct summand of $H^i_{dR}(Y'/W)$ and therefore $V$ is motivic.

		{\itshape $g(C)=1$.}  As in the proof of Lemma 1.4.2 loc. cit., we may tensor $V$ with a torsion rank one connection $L$ over $C$ such that the residues of $V\otimes L$ is nontrivial only at one puncture. As $L$ is motivic, we may assume $U=C-\{x_0\}$ in the following. It is easier for us to work directly with the corresponding local system $\V$ to $V$. Let $\zeta$ be an $r$-th primitive root of unity and $\F_{r,\zeta}$ the rank $r$ local system over $U$ defined in the proof of Lemma 1.4.4 loc. cit. By Proposition 1.4.6 loc. cit., there exists some local system $\L$ over $C$ such that $\V\cong \F_{r,\zeta}\otimes \L|_{U}$. As $\det \V$ is torsion by assumption and $\det \F_{r,\zeta}$ is torsion too by Lemma 1.4.8 loc. cit., so is $\L|_{U}$. Thus it suffices to show $\F_{r,\zeta}$ is motivic. Let $\pi: C\to C$ be a degree $r$ isogeny. Theorem 1.4.16 loc. cit. asserts that there exists a rank one local system $\L_{r,\zeta}$ over $C-\pi^{-1}\{x_0\}$ such that $\pi_*{\L_{r,\zeta}}\cong \F_{r,\zeta}$. It is not difficult to see that one may simply take $\L_{r,\zeta}$ to be the one which has trivial action on generators of $\pi_1(C)$ and whose local monodromy around each point in $\pi^{-1}\{x_0\}$ is the scalar multiplication by $\zeta$. In particular, $\L_{r,\zeta}$ is torsion. Consequently, $\F_{r,\zeta}$ is motivic. By Riemann-Hilbert, the connection $V$ is motivic.

	\end{proof}

\end{document}